\documentclass[hidelinks,onefignum,onetabnum]{siamart220329}

\usepackage{lipsum}
\usepackage{amsfonts}
\usepackage{graphicx}
\usepackage{epstopdf}
\usepackage{algorithm,algpseudocode}
\usepackage{textcomp}
\usepackage{xcolor}
\usepackage[nolist,nohyperlinks]{acronym}
\usepackage{tikz}
\usetikzlibrary{matrix, positioning, 3d, patterns}
\usepackage{multirow}
\usepackage{array}
\usepackage{booktabs}
\usepackage[section]{placeins}
\usepackage{siunitx}
\usepackage{subcaption}
\usepackage{balance}
\usepackage{hyperref}
\usepackage{xspace}

\usepackage{macros}

\ifpdf
  \DeclareGraphicsExtensions{.eps,.pdf,.png,.jpg}
\else
  \DeclareGraphicsExtensions{.eps}
\fi

\Crefname{subsection}{Section}{Sections}
\crefname{subsection}{section}{sections}

\newsiamremark{remark}{Remark}
\newsiamremark{hypothesis}{Hypothesis}
\crefname{hypothesis}{Hypothesis}{Hypotheses}
\newsiamthm{claim}{Claim}

\headers{CSSP for OED}{Eswar, Rao, Saibaba}

\newcommand{\revised}[1]{\textcolor{black}{#1}}

\title{Bayesian D-Optimal Experimental Designs via Column Subset Selection\thanks{
AKS was supported by the National Science Foundation through the award DMS-1845406 and the Department of Energy through the award DE-SC0023188. Vishwas Rao and Srinivas Eswar are supported by the U.S. Department of Energy, Office of Science, Advanced Scientific Computing Research Program under contracts DE-AC02-06CH11357 and DE-SC0023188.}}

\author{Srinivas Eswar\thanks{Mathematics and Computer Science Division, Argonne National Laboratory. \email{seswar@anl.gov}.}
\and Vishwas Rao\thanks{Mathematics and Computer Science Division, Argonne National Laboratory. \email{vhebbur@anl.gov}.}
\and Arvind K.\ Saibaba\thanks{Department of Mathematics, North Carolina State University. \email{asaibab@ncsu.edu}.} }

\acrodef{oed}[OED]{\revised{Optimal Experimental Design}}
\acrodef{map}[MAP]{Maximum a posteriori}
\acrodef{cssp}[CSSP]{Column Subset Selection Problem}
\acrodef{svd}[SVD]{Singular Value Decomposition}
\acrodef{qrcp}[QRCP]{QR with Column Pivoting}
\acrodef{rrqr}[RRQR]{Rank-Revealing QR factorization}
\acrodef{srrqr}[sRRQR]{Strong Rank-Revealing QR factorization}
\acrodef{nla}[NLA]{Numerical Linear Algebra}
\acrodef{matvecs}[matvecs]{matrix-vector products}
\acrodef{gks}[GKS]{Golub-Klema-Stewart}
\acrodef{pde}[PDE]{Partial Differential Equation}
\acrodef{deim}[DEIM]{Discrete Empirical Interpolation Method}
\acrodef{bdeim}[B-DEIM]{Bayesian-Discrete Empirical Interpolation Method}
\acrodef{raf}[RAF-OED]{Randomized Adjoint-free OED}

\makeatletter
\newcommand*{\addFileDependency}[1]{
  \typeout{(#1)}
  \@addtofilelist{#1}
  \IfFileExists{#1}{}{\typeout{No file #1.}}
}
\makeatother

\ifpdf
\hypersetup{
  pdftitle={CSSP for OED},
  pdfauthor={Eswar, Rao, Saibaba}
}
\fi

\begin{document}

\maketitle

\begin{abstract}
This paper tackles optimal sensor placement for Bayesian linear inverse problems, a popular version of the more general \ac{oed} problem, using the D-optimality criterion. This is done by establishing connections between sensor placement and \ac{cssp}, which is a well-studied problem in \ac{nla}. In particular, we use the \ac{gks} approach which involves computing the truncated \ac{svd} followed by a pivoted QR factorization on the right singular vectors. The algorithms are further accelerated by using randomization to compute the low-rank approximation as well as for sampling the indices. The resulting algorithms are robust, computationally efficient, amenable to parallelization, require virtually no parameter tuning, and come with strong theoretical guarantees. One of the proposed algorithms is also adjoint-free which is beneficial in situations, where the adjoint is expensive to evaluate or is not available. 
Additionally, we develop a method for data completion without solving the inverse problem. Numerical experiments on model inverse problems involving the heat equation and seismic tomography in two spatial dimensions demonstrate the performance of our approaches.   
\end{abstract}


\section{Introduction}
\label{sec:intro}
Inverse problems involve reconstructing parameters describing mathematical models from the data. They have wide-ranging applications in medical imaging, geosciences, and other areas of science and engineering.  The inverse problems we tackle are ill-posed and to address this, we use the Bayesian approach to inverse problems which produces a posterior distribution for the reconstruction by combining the likelihood and the prior distributions through the Bayes formula. The posterior distribution captures the uncertainty present in the reconstructions.  

A central problem in Bayesian inverse problems is to determine how to optimally collect data. This falls under the umbrella of \acf{oed}. A common instance of \ac{oed} is to determine the optimal location of a fixed number of sensors. The number of sensors and the locations they can be deployed in are limited due to budgetary and physical constraints. We consider the scenario that we have identified a set of $m$ candidate sensor locations and each sensor contributes to one piece of information; given a budget of $k$ sensors, find the best sensor locations. Put more plainly, choose $k$ ``best'' sensors out of $m$. In \ac{oed}, the notion of optimality is typically defined in terms of a criterion that quantifies the uncertainty in the reconstruction. In this paper, we focus on the D-optimality criterion which is the expected Kullback-Liebler divergence from the prior to the posterior distributions, and is a measure of information gain. For linear inverse problems with Gaussian priors, the D-optimal criterion has a closed-form expression that involves the determinant of the posterior covariance matrix. 

There are several challenges involved in sensor placement and \ac{oed}. Choosing $k$ out of $m$ sensors involves potentially ${m \choose k}$ comparisons, which is infeasible even for modest values of $m$ and $k$, and the optimal selected sensors may not even be unique. To compound these difficulties, even a single objective function evaluation is expensive since it requires many \ac{pde} solves and in the case of the D-optimality criterion, involves expensive determinant computations. There are many approaches in the literature (reviewed in \Cref{sec:prelims}) to address these challenges. 

In this paper, we take a radically different approach and view the problem through the lens of \acf{cssp}, which is a well-studied problem in \ac{nla}. To motivate the connection, we consider a matrix $\M{A}\in \R^{n\times m}$ and consider the problem of choosing the ``best'' $k$ columns of $\M{A}$. As will be discussed in \Cref{sec:prelims}, each column of $\M{A}$ corresponds to a sensor. Motivated by this connection, we develop a computational framework for sensor placement using tools from \ac{cssp} with provable guarantees for the D-optimality criterion and the computational cost.
The algorithms are very efficient and the cost is comparable to computing a truncated \ac{svd} of the forward operator preconditioned by the prior covariance matrix which is known to have low-rank. Our approach also suggests a method for data completion, i.e., filling in the missing values of the sensors at which data is not collected using the data that is actually collected. This data can be used by another inverse problem solution technique, e.g., filtered backpropagation. We briefly survey the contents of this paper and highlight our contributions.

\paragraph{Contributions and Features} This paper makes fundamental contributions to the theory and practice of D-optimal experimental design for Bayesian inverse problems. We summarize the main contributions of this paper and give a brief outline of the paper. Necessary background information is given in \Cref{sec:prelims} including a review of the literature (\Cref{ssec:ip_oed}).
\begin{enumerate}
    \item \textbf{Connections to \ac{cssp}}: In \Cref{sec:cssp}, we interpret sensor placement for D-optimal experimental design through the lens of \ac{cssp}, a framework in \ac{nla} for selecting the ``best'' $k$ columns out of $m$. We also draw connections to \ac{rrqr} and volume-maximization problems. The latter connection also allows us to establish that D-optimal sensor placement is NP-hard, and derive bounds on the performance of greedy algorithms. 
    \item \textbf{Algorithms}: The connection to \ac{cssp} allows us to develop many algorithms. The algorithms are based on the \ac{gks} framework and involve computation of the (truncated) \ac{svd}, followed by subset selection on the right singular vectors $\M{V}_k$ of a matrix $\M{A}$ (the details of this matrix are given in~\Cref{sec:prelims}).  The first set of algorithms uses pivoted QR on $\M{V}_k\t$ (\Cref{ssec:detcssp}). The second algorithm \ac{raf} (\Cref{sec:sensor}) is randomized and is adjoint-free and amenable to parallelization. The 
    final set of algorithms (\Cref{sec:rand}) uses randomized sampling and these approaches are cheaper than \ac{qrcp}.  The proposed algorithms are computationally efficient and require no tuning or input from the user apart from the desired number of sensors $k$.  
    \item \textbf{Analysis}: All the proposed algorithms come with rigorous justification of the bounds on the D-optimality criterion. For the randomized sampling approaches, the bounds hold with high probability. We also provide detailed analyses of the computational cost with complexity estimates. The dominant computational cost is the computation of a truncated SVD of $\M{A}$. Assuming each forward/adjoint operator evaluation is one \ac{pde} evaluation,  the cost of the proposed algorithms are $\sim 4k$ \ac{pde} solves (assuming randomized \ac{svd} with one subspace iteration), and \ac{raf} only requires $\sim k$ \ac{pde} solves.
    \item \textbf{Data completion}: In \Cref{sec:datacomp}, we present an approach called \ac{bdeim} and show how the measured data can be used to estimate the data at the sensor locations at which data is not collected. 
    We derive error analysis for the completed data and the error in the \ac{map} estimate using the completed data. 
    \item \textbf{Numerical experiments}: We perform detailed numerical experiments on model inverse problems involving recovering the initial condition from final time measurements, where the governing dynamics follow the heat equation and seismic tomography in two spatial dimensions. The experiments demonstrate the performance of our algorithms. 
\end{enumerate}
For ease of exposition, the proofs of the results are given in~\Cref{sec:proofs}. Additional background and results are given in \Cref{sec:appendix}.

\section{Preliminaries}
\label{sec:prelims}
This section reviews the notation used in this article and gives background information on Bayesian inverse problems, \ac{oed}, and rank-revealing QR factorizations.  
\subsection{Notation and matrix preliminaries}
Matrices are represented with bold uppercase letters like $\M{A}$ and vectors with bold lowercase letters like $\V{v}$.
\revised{Sets will be represented by uppercase letters like $S$.}
$\M{I}_n$ is the $n\times n$ identity matrix with $\V{e}_j$ as its $j^\text{th}$ column. We denote a selection operator $\M{S} = \bmat{\V{e}_{i_1} & \dots & \V{e}_{i_k}} \in \R^{n\times k}$, which contains columns from the identity matrix corresponding to the \revised{set of indices $S = \{i_1,\dots,i_k\}$}. \revised{We also employ Matlab-style indexing to address elements of vectors and matrices. For example, $\M{A}(:, S) = \M{A}\M{S}$ selects the columns of $\M{A}$ indexed by the set $S$. The complete set of indices is denoted by $\br{n} = \curly{1,\dots,n}$.} The vector of all ones is represented as $\V{e}$.

Let $\|\V{x}\|_2$ denote the 2-norm for a vector $\V{x}$ and $\|\M{A}\|_2$ denote the  spectral norm for a matrix $\M{A}$. If $\M{M}\in \R^{n\times n}$ is positive definite, we define the weighted norm $\|\V{x}\|_{\M{M}} = \sqrt{\V{x}\t\M{M}\V{x}}$. We use $\log$ to denote the natural logarithm and $\logdet(\M{M}) \equiv \log(\det(\M{M}))$, where $\det(\M{M})$ denotes the determinant of the matrix $\M{M}$. \revised{Let $\M{C},\M{D}\in \R^{n\times n}$ be symmetric matrices. We say that $\M{C} \preceq \M{D}$ (alternatively, $\M{D} \succeq \M{C}$) if $\M{D}-\M{C}$ is positive semidefinite. This is known as L\"owner partial ordering and additional properties are given in \Cref{sec:back}.}  

The \ac{svd} of $\M{A} \in \R^{n\times m}$ with $k \le \rank{\M{A}}$ is partitioned
\begin{equation}\label{eqn:svd}\M{A} = \bmat{ \M{U}_k & \M{U}_\perp} \bmat{\M{\Sigma}_k \\ &\M\Sigma_\perp }\bmat{\M{V}_k\t \\ \M{V}_\perp\t} ,\end{equation}
where $\M\Sigma_k\in \R^{k\times k}$ is invertible and contains the dominant singular values of $\M{A}$, and matrices $\M{U}_k\in \R^{n\times k}$ and $\M{V}_k\in \R^{m\times k}$ have orthonormal columns and contain the corresponding left and right singular vectors respectively. Similarly, $\M\Sigma_\perp$ contain the subdominant singular values with corresponding singular vectors in $\M{U}_\perp$ and $\M{V}_\perp$ respectively. For two matrices $\M{A} \in \R^{m\times n},\M{B} \in \R^{n\times p}$ for which the product $\M{AB}$ is defined, we have the singular value inequalities~\cite[Problem III.6.2]{bhatia2013matrix}
\[\sigma_j(\M{AB}) \le \sigma_j(\M{A})\sigma_1(\M{B}) \qquad 1 \le j \le \min\{m,p\}. \]

\subsection{Background on inverse problems and OED}\label{ssec:ip_oed}
Consider the measurement equation 
\begin{equation}\label{eqn:lininv}
\mathbf{d} = \M{F} \V{m} + \V{\epsilon} \,,
\end{equation}
where $\mathbf{d} \in \mathbb{R}^{m}$ is the data, $\mathbf{F} \in \mathbb{R}^{m \times n}$ is the composition of the observation operator and the forward model. The vector $\V{m}$ is the parameter to be reconstructed and is a discretized representation of a spatially-dependent function, so it is very high-dimensional. We assume throughout this paper that $m < n$, that is, the problem is underdetermined. The vector $\M\epsilon$ represents the observation noise, which we assume is Gaussian, that is, $\V{\epsilon} \sim \mathcal{N}\left(0, \M\Gamma_{\rm noise}\right)$. In the bulk of the paper, we consider the uncorrelated noise case $\M\Gamma_{\rm noise} =  \eta^2 \mathbf{I}_m$. An extension to the case of a diagonal covariance matrix $\M\Gamma_{\rm noise}$ is straightforward, but we avoid this to simplify the exposition. An extension to the general uncorrelated  case is left for future work. The inverse problem involves reconstructing the parameters $\V{m}$ from the measurements $\V{d}$. We will be more specific about the measurements when we discuss \ac{oed}.

As this is an ill-posed problem, a prevalent approach is to use the \revised{Bayesian paradigm~\cite{box2011bayesian,reich2019bayesian}. The Bayesian formulation of the inverse problem~\cite{stuart2010inverse,kaipio2005statistical,tenorio2011data,tarantola2005inverse}} first determines the prior information, which in our paper we take to be  Gaussian $\V{m} \sim \mathcal{N}(\V\mu_{\rm pr}, \M\Gamma_{\rm pr})$. Then using Bayes' rule, we derive an expression for the posterior distribution with density $\pi_{\rm post}(\V{m}|\V{d})$ 
\[ \begin{aligned}
    \pi_{\rm post}(\V{m}|\V{d}) = & \> \frac{\pi_{\rm like}(\V{d}|\V{m})\pi_{\rm pr}(\V{m}) }{\pi(\V{d})},\\
    \propto & \> \exp\left(-\frac{1}{2} \|\V{d}-\M{F}\V{m}\|_{\M\Gamma_{\rm noise}^{-1}}^2 -\frac12 \|\V{m}-\V{\mu}_{\rm pr}\|_{\M\Gamma_{\rm pr}^{-1}}^2\right). 
\end{aligned}
\]
Under the assumptions made thus far, the posterior distribution is Gaussian of the form $\mathcal{N}(\V{m}_{\rm post},\M\Gamma_{\rm post})$, where the covariance matrix $\M\Gamma_{\rm post} \equiv ( \M{F}^\Tra\M\Gamma_{\rm noise}^{-1} \M{F} + \M\Gamma_{\rm pr}^{-1})^{-1}$ and the posterior mean is 
\begin{equation}\label{eqn:map}\V{m}_{\rm post} \equiv \M\Gamma_{\rm post}(\M{F}\t (\M\Gamma_{\rm noise}^{-1} \V{d})  + \M\Gamma_{\rm pr}^{-1}\V{\mu}_{\rm pr}). \end{equation}

We consider the sensor placement problem which is one flavor of \ac{oed}. See~\cite{ucinski2005optimal,pukelsheim2006optimal,alexanderian2021optimal,chaloner1995bayesian} for good introductions to this topic.  In this version of the problem, there are $m$ candidate sensor locations $\{\V{x}_1,\dots,\V{x}_m\}$, at which we can consider placing the sensor. But due to budgetary or physical considerations, we can only deploy $k$ sensors and, thus, we would like to pick the ``best'' $k$ locations, out of $m$, to place our sensors.
\revised{This can be encoded as a set of indices $S = \Curly{i_1,\dots,i_k}$ indicating the selected sensors.}
To determine the optimal sensor locations we solve the optimization problem
\begin{equation}\label{eqn:discopt}
    \revised{\min_S~\phi(S),~\text{subject to}\> \abs{S} \le k.}
\end{equation}
\revised{Here $\phi(S)$} is an objective function determining sensor placement \revised{and $S$ is the set of ${m \choose k}$ indices}.
This is a constrained binary optimization problem for determining the sensor locations. Popular criteria for OED selection are the Bayesian A-optimality criterion and the D-optimality criterion which, for linear inverse problems, amounts to computing the trace and log-determinant of the posterior covariance matrices. In this paper, we focus on the D-optimal criterion, which measures the information gain from the prior to the posterior distribution; equivalently, it is the expected Kullback-Liebler divergence from the prior to the posterior distributions~\cite{alexanderian2018efficient}. For the present problem setting, the criterion takes the form 
\revised{
\begin{equation}
\phi(S) =  \logdet(\M{I}+\M{A}(:, S)\M{A}(:, S)\t),
\end{equation}
where the matrix $\M{A}$ is given by
\begin{equation}
    \M{A} \equiv \M\Gamma_{\rm pr}^{1/2}\M{F}\t\M\Gamma_{\rm noise}^{-1/2} \in \R^{n\times m}.
\end{equation}
}
\revised{For simplicity of notation, we define 
\begin{equation}
    \label{eqn:doptfull}
    \revised{\Psi(\M{G})} \equiv \logdet(\M{I}+\M{GG}\t), 
\end{equation}
so that $\phi(S) = \Psi(\M{A}(:,S))$.
}

The columns of $\M{A}$ correspond to the $m$ sensors, and selecting $k$ columns is tantamount to selecting sensors.  This is the key insight that we exploit in this article. Throughout this paper we will assume that the number of sensors $k$ is smaller than the rank of $\M{A}$.
Note that the D-optimality criterion can be expressed in terms of the posterior and the prior covariance matrices as \revised{$\phi([m]) = \Psi(\M{A}) = -\logdet(\M\Gamma_{\rm post}) + \logdet(\M\Gamma_{\rm pr}).$}
\paragraph{Challenges and related work} There are several challenges associated with \ac{oed}: first, forming the posterior covariance matrix explicitly is infeasible, so even a single evaluation of these \ac{oed} criteria is infeasible, and second, the optimization over \revised{the set of sensors} is a constrained binary optimization problem which is challenging to solve. \revised{A detailed review of OED and its challenges is outside the scope of this paper, but further details can be found in~\cite{chaloner1995bayesian,alexanderian2021optimal,huan2024optimal}. }

To tackle the first challenge, one can use randomized techniques to efficiently evaluate the objective functions, which we also adopt in this paper.  There are several strategies to address the second challenge. One approach is to use greedy approaches in which we sequentially select the next best sensor location given a current selection of sensors. Greedy approaches are popular when it comes to sensor placement since they are relatively simple to implement but are typically suboptimal. Justifications for greedy approaches typically rely on the connection to super-modularity (or sub-modularity), see e.g.,~\cite[Section 4.8]{alexanderian2021optimal} and~\cite{jagalur2021batch}. \revised{Another approach is to relax the optimization problem to a continuous optimization problem (see, e.g.,~\cite{haber2008numerical}) with an appropriate regularization term to enforce sparsity. The continuous optimization problem is then solved using gradient-based optimization techniques.} 
Finally,  the optimized weights are then thresholded to obtain binary designs.
Related to this approach, a method that avoided the use of adjoints was proposed in~\cite[Chapter 5]{herman2020design}; see also~\cite[Section 4.5]{alexanderian2021optimal}. 
Another recent approach involves interpreting the binary designs as Bernoulli random variables and expressing the objective function as an expectation, thus allowing the use of stochastic optimization techniques~\cite{attia2022stochastic}. 

In contrast to the previous approaches, this article takes a fresh perspective by viewing \ac{oed} as a problem of selecting $k$ columns from $\M{A}$ out of $m$. This allows us to adapt existing techniques for \ac{cssp}, including high-quality software implementations, to \ac{oed}. The resulting algorithms are efficient, robust, and require virtually no parameter tuning. The performance of the algorithms comes with strong theoretical guarantees. The algorithms rely on low-rank computations of $\M{A}$ which are accelerated using randomized techniques. The closest works to ours are \cite{wu2023fast,wu2023offline}, which also exploit low-rank approximations and swapping algorithms based on leverage scores, but these papers do not explore the connections to \ac{cssp} and give theoretical guarantees. 

\paragraph{Other related work} 
\revised{
While not directly related to \ac{oed}, the notion of algorithmic leveraging~\cite{ma2015statistical} is used to subsample (and then appropriately scale) the data sets and matrices to reduce the computational cost. 
For example, leverage score-based methods have been used to provide scalable regression estimators via sampling~\cite{zhong2023model,ma2015statistical}.
Similarly, a related approach for data reduction is in~\cite{wang2019information}.
However, in contrast, \ac{oed} typically seeks to reduce data because of physical or budgetary constraints, rather than computational considerations.  Another difference compared to the previous approaches is that they do not address the Bayesian setting, which is the focus of this paper.}
\revised{
\ac{cssp} methods have applications beyond compressing a matrix by selecting a subset of its columns \cite{goreinov1997theory}.
It has been used for clustering \cite{boutsidis2014randomized}, data analysis \cite{mahoney2009cur}, feature selection \cite{boutsidis2009unsupervised}, multipoint boundary value problems \cite{de2007subset,de2011note}, image applications \cite{cai2021robust}, graph signal processing \cite{tsitsvero2016signals}, among others.
These works only mention \ac{oed} in passing and do not establish the various connections to \ac{cssp} as in this work.
}

\subsection{Rank-revealing QR factorizations}
The pivoted QR factorization of a matrix $\M{M} \in \Rn{m \times n}$ with $m \ge n$ with $\rank{\M{M}}  = k \le n $ is
\begin{equation}
    \M{M}\M{\Pi} = \M{Q}\M{R} = 
   \bmat{\M{Q}_1 & \M{Q}_2} 
    \begin{bmatrix}
    \M{R}_{11} & \M{R}_{12} \\
            & \M{R}_{22}
    \end{bmatrix},
\label{eqn:qr}
\end{equation}
where $\M{Q} \in \Rn{m \times m}$ is orthogonal, $\M{R}_{11} \in \Rn{k \times k}$ is upper triangular with nonnegative diagonal elements, $\M{R}_{12}\in \Rn{k \times \pr{n-k}}$, $\M{R}_{22} \in \Rn{\pr{m-k}\times\pr{n-k}}$, and $\M{\Pi} \in \Rn{n \times n}$ is a permutation matrix.
The aim is to select a permutation matrix that separates the linearly dependent columns of $\M{M}$ from the linearly independent ones.
This notion was made more precise in Gu and Eisenstat~\cite{gu1996efficient}.

A factorization is called a \acfi{srrqr} if it satisfies the following criteria (for $1\le i \le j$ and $1\le j \le n-k$)
\begin{equation} 
\sigma_i (\M{M}) \ge \sigma_i\pr{\M{R}_{11}} \ge \frac{\sigma_i\pr{\M{M}}}{p_1(k,n)}
\quad\text{and}\quad
\sigma_{k+j}(\M{M}) \le \sigma_j\pr{\M{R}_{22}} \le \sigma_{k+j}\pr{\M{M}}p_1(k,n)
\label{eqn:svals_bounds}
\end{equation}
and
\begin{equation}
\abs{\pr{\M{R}_{11}^{-1}\M{R}_{12}}_{ij}} \le p_2\pr{k,n},
\label{eqn:inv_bounds}
\end{equation}
for $1 \le i \le k$ and and $1 \le j \le n-k$ \cite{gu1996efficient}.
Here $p_1(k,n)$ and $p_2(k, n)$ are functions bounded by low-degree polynomials in $k$ and $n$. The upper bound for $\M{R}_{11}$ and lower bound for $\M{R}_{22}$ hold for any pivoted QR factorization because of interlacing inequalities. 
Gu and Eisenstat~\cite{gu1996efficient} present algorithms that, given $f\ge 1$, find a $\M{\Pi}$ for which \cref{eqn:svals_bounds,eqn:inv_bounds} hold with
\begin{equation}
p_1(k, n) = \sqrt{1 + f^2k(n-k)}
\quad\text{and}\quad
p_2(k, n) = f.
\label{eqn:qbounds}
\end{equation}
When $f > 1$, these methods take $O\pr{\pr{m + n \log_f n}n^2}$ floating point operations (flops).
This reduces to $O(mn^2)$ flops when $f$ is a small power of $n$.
For the short and wide case, i.e. where $m < n$, the operation count is $O(m^2n\log_fn)$ flops~\cite{avron2013faster}.

While the asymptotic complexity of \ac{srrqr} is similar to that of standard QR factorization, it is considerably more complicated to implement efficiently on modern computers.
This is because the best permutation order does not have an optimal substructure, i.e., the order for the best $k$ columns need not contain all of the best $k-1$ columns.
Therefore, a lot of column swapping occurs hurting the efficiency of such algorithms.

The popular \ac{qrcp} was first introduced by Businger and Golub to find a \acl{rrqr} \cite{businger1965linear}.
It is a simple modification of the ordinary QR factorization, where at every step the column with the largest norm in $\M{R}_{22}$ is pivoted to the front and orthogonalized. 
\ac{qrcp} takes about $4mnk - 2k^2\pr{m+n} + 4k^3/3$ flops and is a standard function in many linear algebra packages~\cite{quintana1998blas}.
This algorithm works well in practice, but there are examples where it fails to satisfy \cref{eqn:svals_bounds,eqn:inv_bounds}.
In the worst case, \ac{qrcp} may achieve exponential bounds with $p_1(k, n) = \sqrt{n-k}\cdot2^k$ and $p_2(k, n) = 2^{k-1}$~\cite{gu1996efficient}.

\section{Column subset selection for OED}
\label{sec:cssp}
In \Cref{ssec:oed_cssp}, we first give an interpretation of sensor placement in terms of \ac{cssp}. We give bounds for the D-optimality criterion in \Cref{ssec:struct} which motivates the algorithms in \Cref{ssec:detcssp}. 

\subsection{Interpreting OED as CSSP}\label{ssec:oed_cssp}
Consider the matrix $\M{A}\in \mathbb{R}^{n\times m}$ and recall that the columns of $\M{A}$ correspond to the number of design variables (e.g., sensors). We argue that finding the ``best'' $k$ sensors out of $m$ is closely related to identifying $k$ ``best'' columns of $\M{A}$. 
The D-optimal criterion seeks to maximize 
\begin{equation}\label{eqn:dopt}
\revised{\phi(S) \equiv \Psi(\M{C})   = \logdet(\M{I}+\M{C}\M{C}^\Tra)}
\end{equation} over the set of all matrices $\M{C}$ (containing columns from $\M{A}$ \revised{indexed by the set $S$ with $\abs{S}=k$ such that $\M{C}=\M{A}(:, S)$}; alternatively all $m\times m$ permutation matrices $\M\Pi$ such that $\M{C} = \M{A\Pi}(:,1:k)$). We will discuss interpretations in terms of maximum-volume, \ac{rrqr}, and \ac{gks} approaches; the algorithms we develop are in the \ac{gks} framework.

\paragraph{Connection to maximum-volume} A related problem in \ac{nla} is finding the submatrix of a matrix with maximum volume, which was shown to be an NP-hard problem~\cite{civril2009selecting}. Given a matrix $\M{M}\in \R^{n\times m}$, the volume of a matrix is the volume of the  $n$-dimensional parallelepiped formed by the columns of $\M{M}$; it has the following explicit formulas
\[ \vol(\M{M}) \equiv \sqrt{\det(\M{M}\t\M{M})} = \prod_{j=1}^{\min\{m,n\}}\sigma_j(\M{M}). \]
To see the connection to the maximum volume submatrix problem, observe that by the Sylvester determinant identity \[ \det(\M{I}+\M{AA}\t) = \det(\M{I}+\M{A}\t\M{A}) =   \det(\bmat{\M{I} & \M{A}\t} \bmat{\M{I} \\ \M{A}})  = \vol \left( \bmat{\M{I}\\ \M{A}} \right)^2.\] That is, solving the D-optimal optimization problem is a special case of maximum-volume estimation which is known to be NP-hard (\cite[Theorem 4]{civril2009selecting}). However, that does not readily establish that D-optimal sensor placement is NP-Hard but we give a proof for the sake of completion. On the other hand, classical \ac{oed} is known to be NP-hard~\cite{welch1982algorithmic}. 
\begin{proposition}\label{prop:doptnp}
    \revised{The optimization problem of optimizing the objective function~\eqref{eqn:dopt} over all index sets $S$ corresponding to $k$ columns from $\M{A}$ is NP-hard.}
\end{proposition}
\begin{proof}
    The proof is an extension of the proof technique of~\cite[Theorem 4]{civril2009selecting} and is, therefore, relegated to \Cref{sec:nphard}. 
\end{proof}

This connection to maximum volume can also be used to develop a greedy selection approach to maximize $\vol(\bmat{\M{I} & \M{A}\t})$. The details of this approach are given in~\cite[Algorithm 1]{civril2009selecting}. It can be shown via~\cite[Theorem 11]{civril2009selecting} that the selected columns \revised{$S$ satisfy}
\revised{\[ \phi(S) \le \phi(S^{\rm opt} ) \le 2\log(k!)+ \phi(S). \]}
Here, \revised{$S^{\rm opt}$ denotes the index set corresponding to an optimal set of columns (recall this is not unique).} We give some alternative approaches based on maximum-volume in \Cref{ssec:maxvol}.
\paragraph{Connection to RRQR} To understand the connection of sensor placement to RRQR, consider the pivoted QR factorization of $\M{A}$
\[\M{A} \bmat{\M{\Pi}_1 & \M\Pi_2} = \bmat{\M{C} & \M{T}} = \bmat{\M{Q}_1 & \M{Q}_2} \bmat{\M{R}_{11} & \M{R}_{12} \\  & \M{R}_{22}}.  \]
From the relation $\M{C} = \M{Q}_1\M{R}_{11}$, it is easy to see that $$\logdet(\M{I}+\M{C}\M{C}^\Tra) = \logdet(\M{I}_k+\M{R}_{11}\M{R}_{11}^\Tra).$$ 
Therefore, this relation shows that maximizing 
\revised{$\logdet(\M{I}+\M{C}\M{C}^\Tra)$}
is the same as maximizing $\logdet(\M{I}_k+\M{R}_{11}\M{R}_{11}^\Tra).$ The \ac{srrqr} algorithm attempts to maximize $\det(\M{R}_{11})$ by interchanging  the most ``dependent'' column of $\M{R}_{11}$ (from index $1\le i \le k)$, or the most independent column of $\M{R}_{22}$ (index $k+1 \le i \le n$). It may be possible to modify \ac{srrqr} to instead search for a maximizer of $\logdet(\M{I}_k+\M{R}_{11}\M{R}_{11}^\Tra).$ However, we did not pursue this approach. One reason is that the \ac{srrqr} approach implicitly assumes that the entries of $\M{A}$ are available explicitly. But in the applications we consider, we can only perform \acf{matvecs} with $\M{A}$ (and $\M{A}^\Tra$), so we need to use matrix-free approaches.

\paragraph{The \acf{gks} approach} Instead of \ac{rrqr}, we follow the \ac{gks} approach, which has two stages. In the first stage, a truncated \ac{svd} of $\M{A}\approx \M{U}_k\M\Sigma_k\M{V}_k\t$ is computed, where $\M{V}_k \in \R^{m\times k}$ contains the right singular vectors corresponding to the largest singular values of $\M{A}$. In the second stage, we partition $\M{V}_k\t$ as 
\begin{equation}\label{eqn:vkpart} \M{V}_k\t \bmat{\M\Pi_1 & \M\Pi_2}  = \bmat{\M{V}_{11} & \M{V}_{12}}. \end{equation}
Here $\M\Pi_1\in \R^{n\times k}$ and we define $\M{C} = \M{A\Pi}_1$. 
Assuming that $\M{V}_{11}$ is nonsingular, it can be shown that~\cite[Theorem 5.5.2]{golub2012matrix}
\[
    \frac{\sigma_k(\M{A})}{\|\M{V}_{11}^{-1}\|_2} \le \sigma_k(\M{C}) \le \sigma_k(\M{A}).
\]
Thus, the smallest singular value of $\sigma_k(\M{C})$ is close to $\sigma_k(\M{A})$, except for the factor $\|\M{V}_{11}^{-1}\|_2$. The lower bound clearly identifies the factor $\|\M{V}_{11}^{-1}\|_2$; since $\M{V}_{11}\t$ is an invertible submatrix of $\M{V}_k$ it follows that  $\|\M{V}_{11}^{-1}\|_2 \ge 1$. Therefore, one goal is to identify a permutation matrix $\M\Pi$ such that $\|\M{V}_{11}^{-1}\|_2$ is as close to $1$ as possible; in other words, we want to find a set of $k$ well-conditioned columns of $\M{V}_k\t$. This analysis casts light only on the smallest singular value of $\M{C}$ but the D-optimality criterion involves all the singular values of $\M{C}$, which we address in the next section. In contrast to \ac{rrqr}, \ac{gks} works with the truncated \ac{svd} which can be computed in a matrix-free manner and this is more suitable to applications involving PDEs.

\subsection{Structural bounds on the D-optimality of \texorpdfstring{$\M{C}$}{C}}\label{ssec:struct}

We now derive bounds for the D-optimal criterion, when the columns $\M{C} = \M{A\Pi}_1 = \revised{\M{A}(:,S)}$ have been computed using the \ac{gks} approach. 

\begin{theorem} 
\label{thm:aksbounds}
\revised{Let $\M{A} \in \mathbb{R}^{n\times m}$ with $k \le \rank{\M{A}}$. Then for any permutation $\M{\Pi}$ such that $\rank{\M{V}_{11}}= k$ and  $\M{C} = \M{A}\M{\Pi}_1 = \M{A}(:, S)$ we have, 
\[\Psi(\M{\Sigma}_k/\norm{\M{V}_{11}^{-1}}_2) \le \phi(S) \le \phi(S^{\rm opt}) \le \Psi(\M{\Sigma}_k) \le \phi(\br{m}).\]}
\end{theorem}
\begin{proof}
    See \Cref{ssec:csspproofs}.
\end{proof}
\revised{As before, $S^{\rm opt}$ denotes the indices corresponding to an optimal set of columns and note that $\phi(\br{m}) = \Psi(\M{A})= \logdet(\M{I} + \M{AA}^\Tra)$.}

We now discuss the interpretation of the bounds. \revised{First, consider the  bounds $\phi(S^{\rm opt}) \le \Psi(\M\Sigma_k) \le \phi(\br{m})$, which says that the performance of ``the best'' set of columns in terms of the D-optimal criterion depends on the top-$k$ singular values. 
More precisely, if the singular values don't decay rapidly or the matrix is not approximately rank-$k$, then we cannot expect  
{$\phi(S^{\rm opt})$ to be close to $\phi([m])$.}
In other words, the upper bound is dictated by the singular value decay.
On the other hand, consider the lower bound 
{$\Psi(\M\Sigma_k/\|\M{V}_{11}^{-1}\|_2) = \logdet(\M{I}+\M\Sigma_k^2/\|\M{V}_{11}^{-1}\|_2^2)  \le \phi(S)$.}}
The lower bound clearly identifies the factor $\|\M{V}_{11}^{-1}\|_2$ and as mentioned earlier, the goal is to identify a permutation matrix $\M\Pi$ such that $\|\M{V}_{11}^{-1}\|_2$ is as close to $1$ as possible; in other words, we want to find a set of well-conditioned columns of $\M{V}_k\t$. We show below in \cref{cor:srrqraksbounds} that \ac{srrqr} can be used to identify such a set of columns.

\subsection{Deterministic CSSP algorithm for OED} \label{ssec:detcssp}

In this section, we discuss implementation issues related to the \ac{oed} computed using the \ac{gks} framework. There are two main computational issues to be addressed: first, the computation of the truncated \ac{svd}, and second, the cost of computing the permutation matrix for which we use pivoted QR. The details of the algorithm are given in~\cref{alg:detcssp}. 

\paragraph{Bounds using pivoted QR} We consider the case that pivoted QR is used to compute the permutation matrix $\M\Pi$. The result below quantifies the D-optimal criterion if \ac{srrqr} is used in \cref{alg:detcssp}. 
\begin{corollary}  
\label{cor:srrqraksbounds}
Suppose we select $k$ columns from $\M{A}$, denoted 
\revised{by the set $S$} 
by applying \ac{srrqr} to $\M{V}_k^\Tra$ with factor $f \ge 1$. Let $q_f(m,k) \equiv \sqrt{1+f^2k(m-k)}$.
Then  
$$ 
\revised{\Psi(\M\Sigma_k/q_f(m,k)) \le  \phi(S) \le \phi(S^{\rm opt}) \le \Psi(\M\Sigma_k) \le \phi(\br{m}).}
$$
\end{corollary}
\begin{proof}
    Follows from~\cite[Lemma 2.1]{drmac2018discrete} applied to $\M{V}_k\t$ and \revised{\cref{thm:aksbounds}}.
\end{proof}

In particular, this result shows that the algorithm can find a set of columns that satisfy
\revised{\begin{equation}\label{eqn:existence} \Psi(\M\Sigma_k/\sqrt{1+k(m-k)}) \le  \phi(S)  \le \Psi(\M\Sigma_k) \le \phi(\br{m}),\end{equation}}
but such a set of columns may take exponential time to compute. 

Instead of \ac{srrqr}, we use \ac{qrcp} in practice to compute the permutation matrix $\M\Pi$. While the worst-case behavior for \ac{qrcp} involves the factor $q(m,k) \le \sqrt{m-k}\,2^k$, the algorithm behaves remarkably well in practice. Furthermore, efficient implementations are available from LAPACK \cite{quintana1998blas}.

\paragraph{Computational cost} We perform a truncated \ac{svd} of the preconditioned forward operator to obtain the first $k$ right singular vectors $\M{V}_k$. This will require $\sim k$ applications of the forward and adjoint model each if we employ a matrix-free method for the \ac{svd}  such as a Krylov subspace method~\cite{golub2012matrix} or Randomized \ac{svd}~\cite{halko2011finding}. 
In numerical experiments, we use Randomized \ac{svd} to accelerate the computational cost with a Gaussian random matrix~\cite[Algorithm 5.1]{halko2011finding}. This results in a cost of
\[T_{\text{SVD}} \equiv 2\ell T_{\M{A}} + O(k^2 (m+n) + k^3) \> \text{flops}. \]
Here $T_{\M{A}}$ is the cost of a \ac{matvecs} with $\M{A}$ or $\M{A}\t$ and $\ell = k + p$ where $ p \le 20$ is the oversampling amount.  For additional accuracy, we may also use Randomized SVD with subspace iterations, for some additional computational cost~\cite[Algorithm 5.2]{halko2011finding}. 
The running time for \cref{alg:detcssp} with \ac{srrqr} is $T_{\text{SVD}} + O(k^2m\log_f m)$ flops  (or $T_{\text{SVD}} + O(k^2m)$ flops for \ac{qrcp}). The \ac{matvecs} with $\M{A}$ and $\M{A}\t$ can be parallelized, leading to additional computational benefits.

\begin{algorithm}[!ht]
\caption{OED via Deterministic CSSP}
\begin{algorithmic}[1]
    \Procedure{GKS}{}\newline
    \textbf{Input:} $\M{F} \in \Rn{m \times n}$, prior covariance matrix $\M\Gamma_{\rm pr}$, noise variance $\eta^2$, and number of columns $k$.\newline
    \revised{\textbf{Output:} Set $S$ such that $\M{C} = \M{A}(:, S) \in \Rn{n \times k}$.}
    \State Define operator $\M{A} = \eta^{-1} \M{\Gamma}_{\rm pr}^{1/2}\M{F}\t$.
    \State Compute the truncated SVD: $[\sim, \sim, \M{V}_k] = \mathsf{svd}\pr{\M{A}, k}$.
    \State Run pivoted QR: $[\sim, \sim, \boldsymbol{\Pi}] = \mathsf{qr}\pr{\M{V}_k^\Tra}$. 
    \revised{\State Form $S$ with the first $k$ columns of $\boldsymbol{\Pi}$: $\boldsymbol{\Pi}_1 = \boldsymbol{\Pi}(:, 1:k)$.}
    \EndProcedure
\end{algorithmic}
\label{alg:detcssp}
\end{algorithm}

\section{RAF-OED approach}\label{sec:sensor}
In this section, we develop an algorithm for \ac{oed} called \ac{raf} that is based on the approach described in \Cref{sec:cssp}.

In this approach, we draw a random matrix $\M\Omega \in \R^{d \times n}$ which is a random embedding. For instance, we can draw $\M\Omega$ as a Gaussian random matrix with independent entries drawn from $\mc{N}(0,1/d)$. Next we form the sketch $\M{Y} =\M{\Omega A} \in \R^{d\times m}$, which has the same number of columns as $\M{A}$. We then perform subset selection on $\M{Y}$ by performing pivoted QR directly on $\M{Y}$. This is described in \Cref{alg:randoed}. Justification for this approach of subset selection is given in~\cite{voronin2017efficient,dong2023simpler,duersch2017randomized}. Other choices for the random matrix $\M\Omega$ are possible and discussed in~\cite[Sections 8-9]{martinsson2020randomized}.

\begin{algorithm}[!h]
\caption{OED via RAF-OED}
    \begin{algorithmic}[1]
    \Procedure{RAF-OED}{}\newline
    \textbf{Input:} $\M{F} \in \Rn{m \times n}$, prior covariance matrix $\M\Gamma_{\rm pr}$, noise variance $\eta^2$, number of columns $k$,  and oversampling parameter $p$.\newline
    \revised{\textbf{Output:} Set $S$ such that $\M{C} = \M{A}(:, S) \in \Rn{n \times k}$.}
    \State Draw random matrix $\M{\Omega} \in \R^{d\times n}$, where $d = k + p$, with independent entries drawn from the distribution $\mc{N}(0,1/d)$.
    \State Define operator $\M{A} = \M{\Gamma}_{\rm pr}^{1/2}\M{F}\t \eta^{-1}$ and compute $\M{Y} = \M\Omega\M{A}$.
    \State Run pivoted QR: $[\sim, \sim, \boldsymbol{\Pi}] = \mathsf{qr}\pr{\M{Y}}$.     \revised{\State Form $S$ with the first $k$ columns of $\boldsymbol{\Pi}$: $\boldsymbol{\Pi}_1 = \boldsymbol{\Pi}(:, 1:k)$.}
    \EndProcedure
    \end{algorithmic}
    \label{alg:randoed}
\end{algorithm}

A result similar to~\cref{thm:aksbounds} and \cref{cor:srrqraksbounds} can be stated for \cref{alg:randoed}. 
\begin{theorem}[RAF-OED]\label{thm:randoed} Let $d = k+p$ with $p \ge 2$ and let $\M\Pi_1$ be the output of \cref{alg:randoed} computed using \ac{srrqr} with $f \ge 1$. With probability at least $1-\delta$
\revised{\[ \Psi(\M\Sigma_k/(q_f^R(m,k)C_g)) \le \phi(S) \le  \phi(S^{\rm opt}) \le \Psi(\M\Sigma_k) \le \phi(\br{m}),\]}
where  $C_g \equiv \frac{e\sqrt{d}}{p+1}\left(\frac{2}{\delta}\right)^{1/(p+1)}(\sqrt{n}+\sqrt{d} + \sqrt{2\log(2/\delta)})$.
\end{theorem}
\begin{proof}
    See \Cref{ssec:sensorproofs}.
\end{proof}

\paragraph{Computational cost} \Cref{alg:randoed} requires $dT_{\M{A}}$ matvecs with $\M{A}$ and an additional $\mc{O}(dm\min\{d,m\})$ {flops}.  The value of $d$ can be chosen either as $d= k +p$, where $p$ is a small oversampling parameter. For more robust numerics,  the choice of $d=2k+1$ might be appropriate. A remarkable aspect of this approach is that it only requires matvecs with $\M{A}\t$; that is, a matvec each with $\M{F}$ and $\M\Gamma_{\rm pr}^{1/2}$. Importantly, this completely avoids computing matvecs with $\M{F}\t$. This property is valuable in applications in which adjoint computations with $\M{F}$ are either very expensive or not possible (for example, because of the use of legacy codes). Furthermore, the computations of matvecs with $\M{A}$ can be done in parallel, increasing the computational attractiveness of this method. Computing the \ac{map} estimate in~\eqref{eqn:map} may still require computing adjoints with $\M{F}$, but this can be avoided using alternative techniques such as~\cite{panteleev2015adjoint}.

\section{Randomized subset selection}
\label{sec:rand}
Recall that the \ac{srrqr} algorithm takes $O(k^2m\log_f m)$ flops to select $k$ columns from $m$, while \ac{qrcp} costs $O(k^2m)$ flops. 
To further reduce this cost we use randomized sampling techniques.
This approach involves sampling $s \ge k$ columns using a discrete probability distribution defined based on $\M{V}_k$.
The computational cost of sampling is $O(km)$ which is attractive.
The trade-off, however, is that we need to select $s \sim k \log k$ samples. Since we need to select exactly $k$ columns ($k$ sensors), we propose a hybrid approach (\Cref{ssec:sampling}) that uses both randomization and pivoted  QR factorizations to select columns efficiently. We discuss the analysis in \Cref{ssec:sampanalysis}, and a discussion of computational costs in \Cref{ssec:sampcosts}.

\subsection{Sampling algorithms}\label{ssec:sampling}

The \emph{leverage score} sampling is a well-known approach for column subset selection (see e.g.,~\cite{mahoney2009cur,avron2013faster}).
We first compute the right singular vectors $\M{V}_k$ and define the (subspace) leverage scores as $\tau_j = \|\V{e}_j\t\M{V}_k\|_2^2$ for $1\le j \le m$ which are the squared row norms of $\M{V}_k$. Note that since $\sum_{j=1}^m \tau_j = \|\M{V}_k\|_F^2 = k$, the leverage scores can be used to define a probability distribution on $\{1,\dots,m\}$ with the probability of sampling index $j$ as $\tau_j/k$. 

We  define the sampling probabilities 
\begin{equation}\label{eqn:levscores} \pi_j =  \frac{\tau_j}{2k} + \frac{1}{2m} \qquad 1 \le j \le m,  \end{equation}
which is an average of the leverage score-based distribution and the uniform distribution. The idea is to handle cases where $\tau_j$ is almost zero and to ensure that $\pi_j$ have nontrivial lower bounds~\cite{saibaba2020randomized}.
Note that other distributions, such as $\pi_j = p_j/(\sum_{j=1}^m p_j)$, where $p_j = \max\pr{\tau_j, k/m}$, for $1 \le j \le m$ can also be used instead which achieve similar bounds~\cite[Section 3.4]{avron2013faster}. Another option is to use a mixing parameter $\beta \in (0,1] $ and take the convex combination of the leverage score and the uniform distributions $\pi_j^\beta =  \beta\frac{\tau_j}{k} + (1-\beta)\frac{1}{m}$ for $1\le j \le m$. 

The hybrid approach as described in \cref{alg:leverage_sampling}, works in two stages to return exactly $k$ columns of $\M{A}$. In the first stage, it generates $s = O\pr{k\log k}$ columns of $\M{A}$ chosen based on the probabilities $\{\pi_j\}_{j=1}^m$ chosen independently and with replacement.  The method returns an appropriately weighted set of columns  \begin{equation}\label{eqn:alev}\Mh{C}_{\rm lev} \equiv \M{ASD} \in \R^{n\times s}.\end{equation}
Because we are sampling with replacement, repetition of the columns is allowed.  The reason for weighting the columns is to ensure that $\expect{\left[(\M{ASD})(\M{ASD})\t \right]} = \M{AA}\t$ where  the expectation is over the sampled indices~\cite[Lemma 3]{drineas2006fast}. Similar ideas to the hybrid approach have been proposed earlier in~\cite{boutsidis2009improved,broadbent2010subset}. 

To further pare $\Mh{C}_{\rm lev}$ down to $k$ columns, we perform \ac{srrqr} on $\M{V}_k\t\M{SD}$ to obtain the selection operator $\M\Pi_1 \in \R^{s\times k}$. The corresponding weighted columns of $\M{A}$ are \begin{equation} 
\label{eqn:ahyb} \Mh{C}_{\rm hyb} \equiv \Mh{C}_{\rm lev}\M\Pi_1 = \M{ASD\Pi}_1 \in \R^{n\times k}.\end{equation} 
The final selected indices correspond to the columns of $\M{S\Pi}_1 \in \R^{m\times k}$ and the corresponding selected columns of 
\revised{$\M{A}$ are $\M{C}_{\rm hyb} \equiv \M{AS\Pi}_1 = \M{A}(:, S_{\rm hyb})$}.

A precise analysis of the minimal number of samples is given shortly. 

\begin{algorithm}[ht]
\caption{OED via randomized sampling}
\begin{algorithmic}[1] 
    \Procedure{RandomSampling}{}\newline
    \textbf{Input:} $\M{F} \in \Rn{m \times n}$, prior covariance matrix $\M\Gamma_{\rm pr}$, noise variance $\eta^2$, number of columns $k$, number of samples $s \ge k$.  \newline
    \revised{\textbf{Output:} Set $S$ such that $\M{C} = \M{A}(:, S) \in \Rn{n \times k}$.}
     \State Define operator $\M{A} = \eta^{-1} \M{\Gamma}_{\rm pr}^{1/2}\M{F}\t$.
    \State Compute the truncated \ac{svd}: $[\sim, \sim, \M{V}_k] = \mathsf{svd}(\M{A}, k)$.
    \State Compute the leverage scores in~\eqref{eqn:levscores}.
 
    \State Stage 1: Generate $s$ indices $\{i_1,\dots,i_s\}$ from $\{1,\dots,m\}$ with probabilities $\{\pi_j\}_{j=1}^m$ independently and with replacement.
    \State Define the selection operator $\M{S} = \bmat{\V{e}_{i_1} & \dots & \V{e}_{i_s}}$ and weights $$\M{D} = \diag{\frac{1}{\sqrt{s\pi_{i_1}}}, \dots,\frac{1}{\sqrt{s\pi_{i_s}}}}.$$
    \State Stage 2: Compute \ac{srrqr} with parameter $f\ge 1$ on $\M{V}_k\t\M{SD}$ to obtain $\M\Pi_1 \in \R^{s \times k}$
    \revised{\State Form $S$ with the first $k$ columns of $\boldsymbol{\Pi}$: $\M{\Pi}_1 = \M{\Pi}(:, 1:k)$.}
    \EndProcedure
\end{algorithmic}
\label{alg:leverage_sampling}
\end{algorithm}

\subsection{Analysis of the sampling}\label{ssec:sampanalysis}
We now derive a result on the D-optimal criterion if \cref{alg:leverage_sampling} is used to select the sensors. The result gives insight into the minimal number of samples $s$. 

\begin{theorem}[Random sampling]
\revised{\label{cor:random_unweighted} Let $\epsilon,\delta \in (0,1)$ be user-specified parameters and suppose  $s \ge 4k \epsilon^{-2}\log(k/\delta)$.   Define $\M{C}_{\rm lev} = \M{AS} = \M{A}(:, S_{\rm lev})$ and $\M{C}_{\rm hyb} = \M{AS\Pi}_1 = \M{A}(:, S_{\rm hyb})$ as the unweighted columns generated using \cref{alg:leverage_sampling}. Then with probability at least $1-\delta$, for the 
\[ \Psi(\M\Sigma_k/  (q_f^U(m,s,k) ) ) \le \phi(S_{\rm hyb}) \le \phi(S_{\rm lev}) ,\]
where $q_f^U(m,s,k) \equiv q_f(s,k)\sqrt{(2m/s(1-\epsilon))}$ .}
\end{theorem}
\begin{proof}
See \Cref{ssec:random_unwproofs}
\end{proof}
Since $\M{C}_{\rm lev}$ chooses $s \ge k$ columns, \revised{the upper bound of $\Psi(\M\Sigma_k)$ no longer holds for $\phi(S_{\rm lev})$, but it is true that $\phi(S_{\rm hyb}) \le \Psi(\M\Sigma_k)$.}

\subsection{Summary of computational costs}\label{ssec:sampcosts} Let us look at the savings in flops achieved via randomized sampling instead of running a QR factorization. Computing the sampling probabilities $\{\pi_j\}_{j=1}^m$ requires one pass over $\M{V}_k$, and costs $O(mk)$ flops. Sampling and forming the selection matrix requires $O(s)$ flops. The cost of the second stage requires an additional pivoted QR factorization performed on an $k \times s$ matrix.
This takes $O(k^2s\log_f s)$ flops (\ac{srrqr}) or $O(k^2s)$ flops (\ac{qrcp}) depending on the algorithm used.

\section{Data completion}
\label{sec:datacomp}
Suppose we are given a limited set of measurements $\M{S}\t \V{d} \in \mathbb{R}^{k}$, that collects data at sensors as determined by the selection operator $\M{S}$. We want to ``complete'' the measurements to provide the data at all the candidate measurement locations. The completed data can then be used to solve an inverse problem by computing the \ac{map} estimate, or using another inversion technique.

To this end, suppose we are given the right singular vectors $\M{V}_k$ of the preconditioned operator $\M{A}$. We further assume that $\M{S}^\Tra \M{V}_k$ is invertible, and we can define the oblique projector $\MB{P} \equiv \M{V}_k(\M{S}^\Tra \M{V}_k)^{-1}\M{S}^\Tra$ (notice $\MB{P}^2 = \MB{P}$, so it is idempotent).
We propose the following approximation to $\V{d}$  through the formula 
\begin{equation}\label{eqn:datacomplete}
    \MB{P}\V{d} = \M{V}_k(\M{S}^\Tra \M{V}_k)^{-1} (\M{S}^\Tra\V{d}).
\end{equation} 
Notice that in the expression above we only select data at $\M{S}\t\V{d}$.  This approach uses a judicious combination of the forward operator, prior and noise covariance matrices to generate the basis $\M{V}_k$. The projector $\MB{P}$ is an interpolatory projector~\cite[Definition 3.1]{sorensen2016deim} and has the interpolatory property $\M{S}\t\MB{P}\V{d} = \M{S}\t\V{d}$. This implies that $\MB{P}\V{d}$ matches $\V{d}$ {\em exactly} (in exact arithmetic) at the indices determined by $\M{S}$. A similar approach has been used in model reduction and is called the \acf{deim} approach \cite{chaturantabut2010nonlinear}. We call our approach \acf{bdeim}, to recognize the fact that the basis arises from the forward operator preconditioned by the square roots of the noise and the prior covariance matrices. Note that this approach gives an inexpensive way to complete the data (assuming the singular vectors $\M{V}
_k$ are available) without having to solve the inverse problem explicitly.

We now derive a result on the squared absolute error in the completed data. 
\begin{theorem}[\ac{bdeim} error]\label{thm:complete} Let $\M\Gamma_{\rm noise} = \eta^2\M{I}$ and let  $1 \le k < m$. Let $\MB{P}\V{d}$ be the completed data as in~\eqref{eqn:datacomplete}. The absolute error in the completed data satisfies
\[ \expect_{\V{d}}\left[\|(\M{I}-\MB{P})\V{d}\|_{\M\Gamma_{\rm noise}^{-1}  }\right] \le  \|(\M{S}^\Tra \M{V}_k)^{-1}\|_2 \left[ \|\M\Sigma_\perp\|_F +    \|\M\Sigma_\perp\|_2 \|\V\mu_{\rm pr}\|_{\M\Gamma_{\rm pr}^{-1}} +  \sqrt{m-k}\right].\] 
The expectation is taken over the data, with density $\pi(\V{d}).$ 
\end{theorem}
\begin{proof}
    See \Cref{ssec:datacompproofs}.
\end{proof}

We now discuss an interpretation of this theorem. The absolute error in the completed data is small, in expectation, if the singular values in $\M\Sigma_\perp$ decay rapidly, and the amplification factor $\|(\M{S}\t \M{V}_k)^{-1}\|_2 $ is small. 
If \ac{srrqr} is used with parameter $f\ge 1$ to determine the selection operator $\M{S}$ as in \cref{alg:detcssp}, then the bound takes the form 
\begin{equation}\label{eqn:comp_srrqr} \expect_{\V{d}}\left[\|(\M{I}-\MB{P})\V{d}\|_{\M\Gamma_{\rm noise}^{-1}  }\right] \le  q_f(m,k) \left[ \|\M\Sigma_\perp\|_F +   \|\M\Sigma_\perp\|_2 \|\V\mu_{\rm pr}\|_{\M\Gamma_{\rm pr}^{-1}} + \sqrt{m-k}\right].\end{equation} 
It should also be noted that the completed data $\MB{P}\V{d}$ has a noise covariance matrix that is no longer uncorrelated; in particular the noise covariance is changed to $\MB{P}\M\Gamma_{\rm noise}\MB{P}\t = \eta^2 \M{V}_k (\M{S}\t\M{V}_k)^{-1}(\M{S}\t\M{V}_k)^{-\Tra}\M{V}_k\t$.

\paragraph{Approximate MAP estimate} The completed data $\MB{P}\V{d}$ can be used to define an approximate \ac{map} point 
\[ \widehat{\V{m}}_{\rm post} \equiv \M\Gamma_{\rm post}( \M{F}\t(\M\Gamma_{\rm noise}^{-1} \MB{P}\V{d}) + \M\Gamma_{\rm pr}^{-1}\V\mu_{\rm pr}). \]
The following corollary quantifies the error in the approximate \ac{map} point, which shows that the absolute error in the \ac{map} point has the same upper bound as that of the absolute error in the completed data. 

\begin{corollary}
    \label{cor:approxmap}
 The  error in the \ac{map} point in the $\|\cdot\|_{\M\Gamma_{\rm pr}^{-1}}$ norm computed using the \ac{bdeim} approach satisfies 
\[ \expect_{\V{d}}\left[\|\widehat{\V{m}}_{\rm post} - {\V{m}}_{\rm post} \|_{\M\Gamma_{\rm pr}^{-1}} \right] \le  q_f(m,k) \left[ \|\M\Sigma_\perp\|_F +  \|\M\Sigma_\perp\|_2 \|\V\mu_{\rm pr}\|_{\M\Gamma_{\rm pr}^{-1}} + \sqrt{m-k}\right] . \]   
\end{corollary} 
\begin{proof}
    See \Cref{ssec:datacompproofs}.
\end{proof}

\section{Proofs}\label{sec:proofs}

This section collects the proofs from various points in this article. Each subsection is dedicated to covering the proofs from a specific section. Some additional background in \Cref{sec:back} might help in understanding the proofs. 
\subsection{Proofs of \Cref{sec:cssp}}\label{ssec:csspproofs}
\begin{proof}[Proof of \cref{thm:aksbounds}]
Compute the pivoted QR factorization of $\M{V}_k\t$ as in~\eqref{eqn:vkpart}.  By assumption, $\M{V}_{11} \in \Rn{k \times k}$ is nonsingular. 
Similarly partition $\M{A}\M{\Pi} = \bmat{\M{C} & \M{T}}$. 
Write the matrix $\M{C}$ as 
$$
\M{C} = \M{A}\M{\Pi}_1 
= 
\M{U}\M{\Sigma}\M{V}^\Tra \M{\Pi}_1 
= 
\M{U} \bmat{\M\Sigma_k \M{V}_k^\Tra \M\Pi_1 \\ \M\Sigma_\perp \M{V}_\perp^\Tra\M{\Pi}_1 }
=
\M{U} \bmat{\M\Sigma_k \M{V}_{11}\\ *}
.
$$
Multiplication by an orthogonal matrix does not change the singular values, so
$$
\sigma_j(\M{C}) = \sigma_j \left( \bmat{\M\Sigma_k \M{V}_{11} \\ * }\right)  \ge \sigma_j (\M{\Sigma}_k \M{V}_{11})\qquad 1 \le j \le k.
$$
We have used the fact that the singular values of a submatrix are smaller than the singular values of the matrix in which it is contained. 
Next, write $\M\Sigma_k = \M\Sigma_k  \M{V}_{11}\M{V}_{11}^{-1}$,  and apply the singular value inequalities~\cite[Problem III.6.2.]{bhatia2013matrix}
$$
\sigma_j (\M\Sigma_k) \le   \sigma_j (\M\Sigma_k  \M{V}_{11}) \norm{\M{V}_{11}^{-1}}_2, \qquad 1 \le j \le k.
$$
Combine with the previous inequality to get 
$$
\frac{\sigma_j(\M{A}) }{\norm{\M{V}_{11}^{-1}}_2}  \le \sigma_j(\M{C}) \le \sigma_j(\M{A}), \qquad 1 \le j \le k.  
$$ 
Here $\sigma_j(\M{A}) = \sigma_j(\M\Sigma_k)$ are the singular values of $\M{A}$. 
The upper bound follows from applying the singular value inequality directly to $\M{C}$. 
The desired bound follows from the properties of the $\logdet$.
\end{proof}

\subsection{Proofs of \Cref{sec:sensor}}\label{ssec:sensorproofs}
\begin{proof}[Proof of \cref{thm:randoed}]
    Only the first inequality needs to be proven, since the rest follows from \cref{thm:aksbounds}. We first prove that $\M{Y}$ has at least rank $k$ with probability $1$ by deriving lower bounds for its singular values.

    Given the partitioned SVD of $\M{A}$~\eqref{eqn:svd}, we can write 
    \[ \M{Y} = \M\Omega \M{A} = \bmat{ (\M{\Omega U}_k)\M\Sigma_k & (\M\Omega\M{U}_\perp) \M\Sigma_\perp}\M{V}\t.   \]
    The matrix $\M\Omega\M{U}_k \in \R^{d \times k}$ has the same distribution as a $d\times k$ Gaussian random matrix with independent entries drawn from $\mc{N}(0,1/d)$, since the Gaussian distribution is rotationally invariant. Therefore, this matrix has full column rank with probability $1$, so it has a left multiplicative inverse. From $\M\Sigma_k = (\M{\Omega U}_k)^\dagger (\M{\Omega U}_k) \M\Sigma_k$, we have 
    \[ \sigma_j(\M{Y}) \ge \sigma_j ( (\M{\Omega U}_k)\M\Sigma_k) \ge \frac{\sigma_j(\M\Sigma_k)}{ \|(\M{\Omega U}_k)^\dagger\|_2} = \frac{\sigma_j(\M{A})}{ \|(\M{\Omega U}_k)^\dagger\|_2} \qquad 1 \le  j \le k. \]
    Since $\rank{\M{A}} \ge k$, this shows $\rank{\M{Y}} \ge k$ with probability $1$. We can then apply \ac{srrqr} on $\M{Y}$ with parameter $f \ge 1$. Then 
    \[ \M{Y} \bmat{\M\Pi_1 & \M\Pi_2} = \bmat{\M{Q}_1 & \M{Q}_2} \bmat{\M{R}_{11} & \M{R}_{12} \\ & \M{R}_{22}}.\]
   By the guarantees of \ac{srrqr} $\sigma_j(\M{Y\Pi}_1) \ge \sigma_j(\M{Y})/q_f(m,k)$ for $1 \le j \le k$. Finally, the singular value inequalities imply $\sigma_j(\M{Y\Pi}_1) \le \|\M\Omega\|_2 \sigma_j(\M{A\Pi}_1)$ for $1 \le j \le k$. Putting all the intermediate steps together, we get 
    \[ \sigma_j(\M{A\Pi}_1) \ge \frac{\sigma_j(\M{Y\Pi}_1)}{ \|\M\Omega \|_2 }  \ge \frac{\sigma_j(\M{Y})}{q_f(m,k)\|\M\Omega \|_2 } \ge \frac{\sigma_j(\M{A})}{q_f(m,k) \|\M\Omega \|_2 \|(\M{\Omega U}_k)^\dagger\|_2}  \qquad 1 \le j \le k.\] 
     By a simple modification of the proof technique of~\cite[Theorem 5.8]{gu2015subspace}, we can show with probability at least $1-\delta$, $\|\M\Omega \|_2 \|(\M{\Omega U}_k)^\dagger\|_2\le C_g$. For completeness, we give the details in \Cref{ssec:proofdetails}. 
    The bound is completed by plugging in the bound for $\|\M\Omega \|_2 \|(\M{\Omega U}_k)^\dagger\|_2$ and using the properties of the log-determinant.
\end{proof}

\subsection{Proofs of \Cref{sec:rand}}
\label{ssec:random_unwproofs}

\begin{proof}[Proof of \cref{cor:random_unweighted}]
Since the columns of $\M{C}_{\rm hyb}$ is a subset of those in $\M{C}_{\rm lev}$, follows that  $\M{C}_{\rm hyb}\M{C}_{\rm hyb}\t \preceq \M{C}_{\rm lev} \M{C}_{\rm lev}\t$, 
\revised{so $\phi(S_{\rm hyb}) \le \phi(S_{\rm lev})$.} 
From \cref{thm:aksbounds} with the selection operator $\M{S\Pi}_1$, we see that 
\revised{
\[ \Psi(\M\Sigma_k/\|(\M{V}_k\t\M{S\Pi}_1)^{-1}\|_2 ) \le \phi(S_{\rm hyb}).\] 
}
We then need a lower bound for $\|(\M{V}_k\t\M{S\Pi}_1)^{-1}\|_2$. Let $\M{D}_1\in \R^{k\times k}$ be a diagonal matrix such that $\M{SD\Pi}_1 = \M{S\Pi}_1\M{D}_1$; this diagonal matrix $\M{D}_1$ is a submatrix of $\M{D}$ corresponding to the entries of $\M\Pi_1$. 
    
Recall $\Mh{S} = \M{SD\Pi}_1$. First, we show $\M{V}_k\t\Mh{S}$ is invertible with high probability and derive a bound for $\|(\M{V}_k\t\Mh{S})^{-1}\|_2.$ From~\cite[Theorem 6.2]{holodnak2015randomized}, with a probability at least $1-\delta$, we have $\sigma_k(\M{V}_k\t\M{SD}) \ge \sqrt{1-\epsilon}$. Applying \ac{srrqr} with parameter $f \ge 1$, we have 
\begin{equation}
    \label{eqn:levinter}
    \frac{\sigma_k(\M{V}_k\t\M{SD})}{q_f(s,k)} \le \sigma_k(\M{V}_k\t\M{SD\Pi}_1), 
\end{equation} 
so that with probability at least $1-\delta$, $\|(\M{V}_k\t\Mh{S})^{-1}\|_2 \le q_f(s,k)/\sqrt{1-\epsilon}$ .  
Furthermore,  $\|\M{D}_1\|_2 \le \|\M{D}\|_2 \le \max_{1\le j \le m} (1/\sqrt{s\pi_j}) \le \sqrt{(2m/s)}$. Putting these intermediate results together, we have probability at least $1-\delta$
    \begin{equation}\label{eqn:unweightedinter}
        \|(\M{V}_k\t\M{S\Pi}_1)^{-1}\|_2 \le \|(\M{V}_k\t\M{SD\Pi}_1)^{-1}\|_2\|\M{D}\|_2 \le q_f^U(m,s,k).
    \end{equation}  
    We have used $(\M{V}_k\t\M{S\Pi}_1)^{-1} = \M{D}_1 (\M{V}_k\t\M{S\Pi}_1\M{D}_1)^{-1} = \M{D}_1 (\M{V}_k\t\M{SD\Pi}_1)^{-1}$ and $\|\M{D}_1\|_2 \le \|\M{D}\|_2$. This establishes the bound.  
\end{proof}

\subsection{Proofs of \Cref{sec:datacomp}}\label{ssec:datacompproofs}

\begin{proof}[Proof of \Cref{thm:complete}]
From $\MB{P}\M{V}_k\M{V}_k\t = \M{V}_k\M{V}_k\t$ , we have $(\M{I} - \MB{P}) = (\M{I} - \MB{P})(\M{I} - \M{V}_k\M{V}_k\t).$ Therefore, using submultiplicativity 
\[ \begin{aligned}\|(\M{I}-\MB{P})\V{d}\|_2 \le & \> \|\M{I} - \MB{P}\|_2 \| ( \M{I} - \M{V}_k\M{V}_k\t)\V{d}\|_2 \\ = & \> \|(\M{S}\t \M{V}_k)^{-1}\|_2   \| ( \M{I} - \M{V}_k\M{V}_k\t)\V{d}\|_2. \end{aligned}\] 
In moving from the first equation to the second, we have used~\cite[Theorem 1]{szyld2006many} which says $\|\M{I} - \MB{P}\|_2 = \|\MB{P}\|_2$. The theorem applies since $1 \le k < m$, so $\MB{P} \neq \M{0}, \M{I}$ and, furthermore, $\|\MB{P}\|_2 = \|(\M{S}\t \M{V}_k)^{-1}\|_2$ since the columns of $\M{V}_k$ and $\M{S}$ are orthonormal. Taking expectations 
\begin{equation}\label{eqn:exinter1}\expect_{\V{d}}\left[\|(\M{I}-\MB{P})\V{d}\|_2^2\right] \le \|(\M{S}\t \M{V}_k)^{-1}\|_2^2 \expect_{\V{d}}\left[ \| ( \M{I} - \M{V}_k\M{V}_k\t)\V{d}\|_2^2\right]. \end{equation}
Using the tower law of conditional expectations, we get 
\[ \expect_{\V{d}}\left[ \| ( \M{I} - \M{V}_k\M{V}_k\t)\V{d}\|_2^2\right] =  \expect_{\V{m}}\left\{\expect_{\V{d}|\V{m}}\left[ \| ( \M{I} - \M{V}_k\M{V}_k\t)\V{d}\|_2^2\right]\right\}.\] 
We tackle the inner expectation first. Note that if $\V{x} \sim \mathcal{N}(\V\mu, \M\Gamma)$, then \begin{equation}\label{eqn:expectsq}\expect[\|\V{x}\|_2^2] = \trace(\M\Gamma) + \|\V{\mu}\|_2^2.\end{equation}
Using $\V{d}|\V{m} \sim \mathcal{N}(\M{F}\V{m}, \M\Gamma_{\rm noise})$  and~\eqref{eqn:expectsq} we get
\[ \begin{aligned} \expect_{\V{d}|\V{m}}\left[ \| ( \M{I} - \M{V}_k\M{V}_k\t)\V{d}\|_2^2 \right]= & \>  \trace( ( \M{I} - \M{V}_k\M{V}_k\t)\M\Gamma_{\rm noise}( \M{I} - \M{V}_k\M{V}_k\t) ) \\
& \qquad + {\|( \M{I} - \M{V}_k\M{V}_k\t)\M{F}\V{m}\|_2^2}. \end{aligned}\]
 Since $\M\Gamma_{\rm noise} = \eta^2 \M{I}$,  and $\M{I} - \M{V}_k\M{V}_k\t$ is an orthogonal projector, the trace term simplifies to $\eta^2 (m-k)$. Therefore, 
\[  \expect_{\V{m}}\left\{\expect_{\V{d}|\V{m}}\left[ \| ( \M{I} - \M{V}_k\M{V}_k\t)\V{d}\|_2^2\right]\right\}=  \eta^2 (m-k) + \expect_{\V{m}}\left[ \| ( \M{I} - \M{V}_k\M{V}_k\t)\M{F}\V{m}\|_2^2\right]. \]
The expectation can be computed explicitly once again using~\eqref{eqn:expectsq} as  
\[\begin{aligned}\expect_{\V{m}}\left[ \| ( \M{I} - \M{V}_k\M{V}_k\t)\M{F}\V{m}\|_2^2\right] = & \>  \trace(( \M{I} - \M{V}_k\M{V}_k\t)\M{F\Gamma}_{\rm pr} \M{F}\t( \M{I} - \M{V}_k\M{V}_k\t) )  \\
& \> \qquad +  \|(\M{I} - \M{V}_k\M{V}_k\t)\M{F}\V\mu_{\rm pr}\|_2^2 \\
= & \eta^2 \trace(( \M{I} - \M{V}_k\M{V}_k\t)\M{A}\t\M{A}( \M{I} - \M{V}_k\M{V}_k\t) ) \\
& \> \qquad + \eta^2 \|(\M{I} - \M{V}_k\M{V}_k\t)\M{A}\t\M\Gamma_{\rm pr}^{-1/2}\V\mu_{\rm pr}\|_2^2   \\
\le & \>  \eta^2 \|\M\Sigma_\perp\|_F^2 + \eta^2 \|\M\Sigma_\perp\|_2^2 \|\V\mu_{\rm pr}\|_{\M\Gamma_{\rm pr}^{-1}}^2 .  \end{aligned}\]
Finally, this gives $\expect_{\V{d}}\left[ \| ( \M{I} - \M{V}_k\M{V}_k\t)\V{d}\|_2^2\right] = \eta^2 [ \|\M\Sigma_\perp\|_F^2  + \eta^2 \|\M\Sigma_\perp\|_2^2 \|\V\mu_{\rm pr}\|_{\M\Gamma_{\rm pr}^{-1}}^2+  (m-k)]$.  
Plug into~\eqref{eqn:exinter1}. The bound follows by using the Cauchy-Schwartz inequality and subadditivity of square roots.
\end{proof}

\begin{proof}[Proof of \Cref{cor:approxmap}]
    Simple algebraic manipulations show
  \[  \begin{aligned}
        \M\Gamma_{\rm pr}^{-1/2}(\V{m}_{\rm post} - \widehat{\V{m}}_{\rm post})   = & \>  \eta^{-1} \M\Gamma_{\rm pr}^{-1/2} \M\Gamma_{\rm post}\M{F}\t (\eta^{-1}(\V{d}-\MB{P}\V{d} ) ) \\
         = & \>  (\M{AA}\t + \M{I})^{-1}\M{A} [\eta^{-1}(\V{d}-\MB{P}\V{d} ) ].
    \end{aligned}\]
   Apply submultiplicativity and~\eqref{eqn:comp_srrqr}. The bound follows from $\|(\M{AA}\t + \M{I})^{-1}\M{A}\|_2 \le 1$.
\end{proof}

\section{Experimental results}
\label{sec:exp}
In this section, we demonstrate numerically the efficacy of proposed methods\footnote{Code to reproduce our results is available from \url{https://github.com/RandomizedOED/css4oed} (see also \Cref{sec:codeloc}).}.

\begin{figure}[!ht]
    \centering
    \small
    \includegraphics[width=\textwidth]{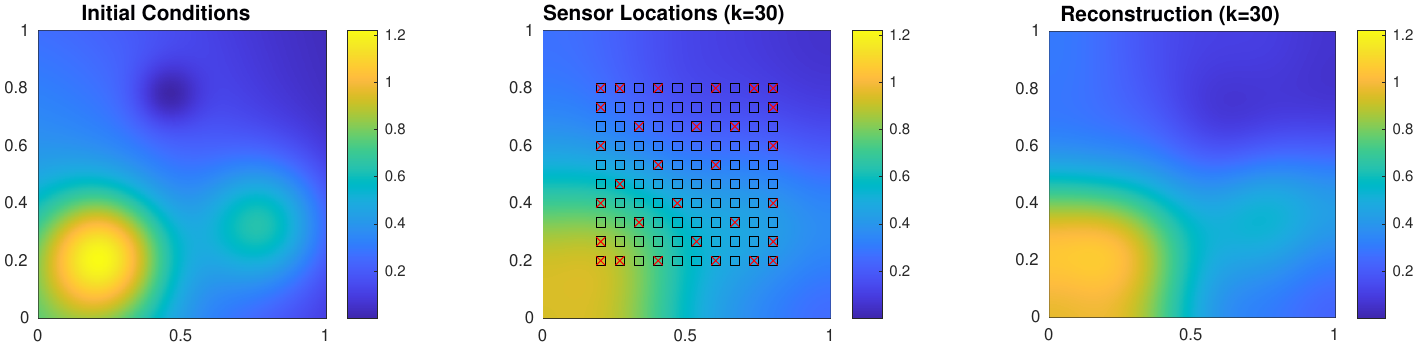}
    \caption{Problem setup for the Heat problem. The left panel shows the true initial conditions (Franke's function). The middle panel shows the true state along with the sensor locations as black squares. The red sensors are the ones selected by RandGKS. The right shows the reconstruction from these 30 sensors.}
    \label{fig:diffusionsetup}
\end{figure}

We conducted our numerical experiments on two different test problems from the AIR Tools II package~\cite{hansen2018air} through the IR Tools interface~\cite{gazzola2019ir}.

\begin{figure}[!ht]
    \centering
    \small
    \includegraphics[width=\textwidth]{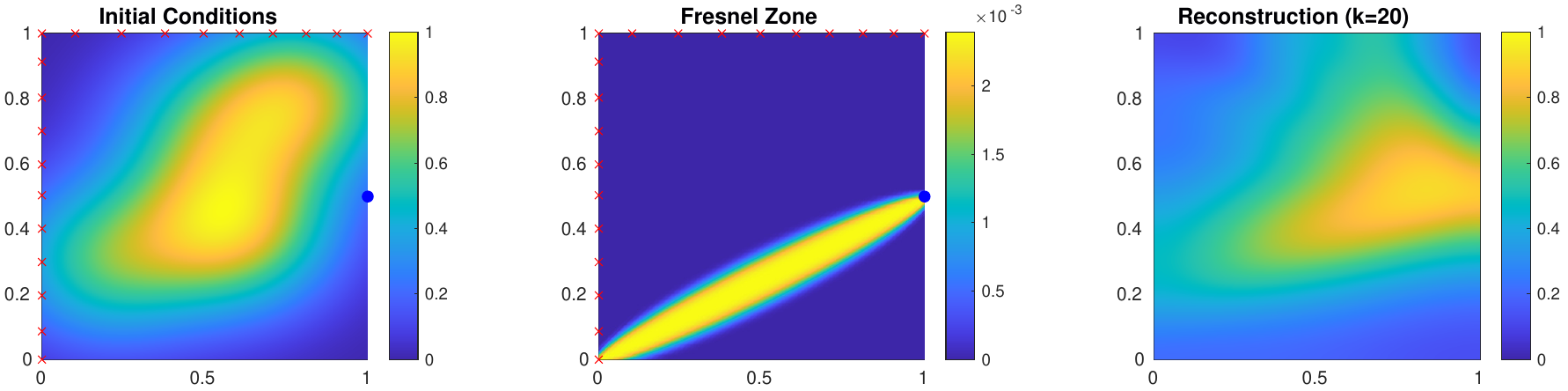}
    \caption{Problem setup for the seismic tomography problem. The left panel shows the true image along with the blue source on its right boundary. 256 receivers are uniformly placed on the left and top boundaries (a subset of receivers are denoted by red stars). The 20 red locations are selected by RandGKS. A Fresnel zone for one selected source-receiver pair is shown in the middle panel. The right panel contains the reconstruction from the selected sensors.} 
    \label{fig:seismicsetup} 
\end{figure}
  \paragraph{Application 1: Heat Equation} In this test problem, we consider a 2D heat equation
        \[ \frac{\partial u}{\partial t} = \Delta u \qquad \V{x} \in (0,1)^2, \]
    with homogeneous Neumann boundary conditions. 
    The domain is discretized into a grid of size $65^2 = 4225$ degrees of freedom. 
    The PDE is discretized using linear finite elements. 
    Data is collected at a discrete set of $100$ measurement locations at a final time $T = 0.01$ after $100$ time steps, and the inverse problem is to recover the initial condition.  A plot of the true initial condition, the candidate sensor locations, and the reconstruction is given in \cref{fig:diffusionsetup}. 
    \paragraph{Application 2: Seismic travel-time tomography} The goal in this form of tomography is to determine sub-surface attenuation in a domain shown as an $N\times N$ image represented by the flattened vector $\V{x} \in \Rn{n}$ with $n = N^2$ (see \cref{fig:seismicsetup}). 
    The domain is discretized into a grid of size $65^2 = 4225$ with a single source on the right boundary.
    Furthermore, $256$ receivers are located on the left and top boundaries with uniform spacing. The waves are assumed to travel between the source and receiver within the first Fresnel zone, which is shaped as an ellipse with focal points at the source and receiver (see middle panel of \cref{fig:seismicsetup}).
    The travel time is modeled as an integration of attenuation coefficients over the Fresnel zone.
    This line integral is discretized using a sensitivity kernel. 

We now briefly describe other problem settings. To simulate measurement error, $2\%$ noise is added to the data. We utilize the same prior distributions for both test problems. The prior mean is taken to be zero and the prior covariance matrix is taken to be $\M\Gamma_{\rm pr}^{-1} = \alpha \M{KM}^{-1}\M{K}$, where $\M{K}$ is a finite element stiffness matrix corresponding to the operator $(-\Delta + \kappa^2)$ with Neumann boundary conditions, and $\M{M}$ is the mass matrix. 
We use the following values for the parameters, $\kappa^2 = 80$ and $\alpha = 0.1$. 
A justification for this form of the prior covariance matrix is given in~\cite{bui2013computational}. We report the relative error of the reconstruction in the 2-norm, i.e., $\|\V{m} - \widehat{\V{m}}\|_2/ \|\V{m}\|_2$.

\subsection{Varying number of sensors}
\label{sec:exp_varyk}
We first consider the performance of the RandGKS algorithm as we increase the number of sensors selected. \textbf{RandGKS} is an implementation of \cref{alg:detcssp} where a randomized \ac{svd} algorithm (with $q=1$ subspace iterations and oversampling $p=20$) is used to compute the right singular vectors of $\M{A}$  and \ac{qrcp} is used for computing the pivot indices that determine the sensor placement. We compare RandGKS against the performance of the full operator (Full) in which all sensors are selected.

\begin{figure}[!ht]
    \centering
    \includegraphics[scale=0.5]{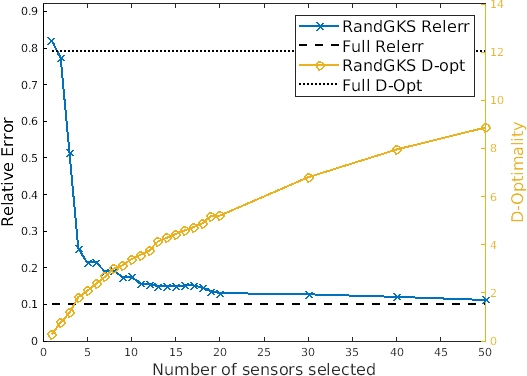}
    \caption{D-optimality and relative error changes in the RandGKS algorithm with an increasing number of sensors $k$ for the Heat problem. The relative error decreases and the D-optimal criterion increases with increasing $k$. Notice that we achieve good reconstructions with as few as $20$ sensors with respect to the full operator but D-optimality is smaller than that of the full operator.}
    \label{fig:diffusionvaryk}
\end{figure}

\Cref{fig:diffusionvaryk} shows both the relative error and D-optimal criterion as we vary $k$  from $1$ to $50$ for the Heat problem. On the one hand, moving from left to right, the relative error shows a sharp decrease with an increasing number of sensors $k \le 20$, and then plateaus which suggest diminishing returns. On the other hand, the D-optimality increases with increasing number of sensors.  
The RandGKS method can achieve reasonably good reconstructions with as few as $k=20$ sensors when compared with Full; however, RandGKS yields a much smaller D-optimal solution compared to Full. Even though there is no improvement in accuracy in terms of relative error, additional information collected from sensors can help reduce the uncertainty (i.e., increased D-optimality). 

\begin{figure}[!ht]
    \centering
    \includegraphics[scale=0.35]{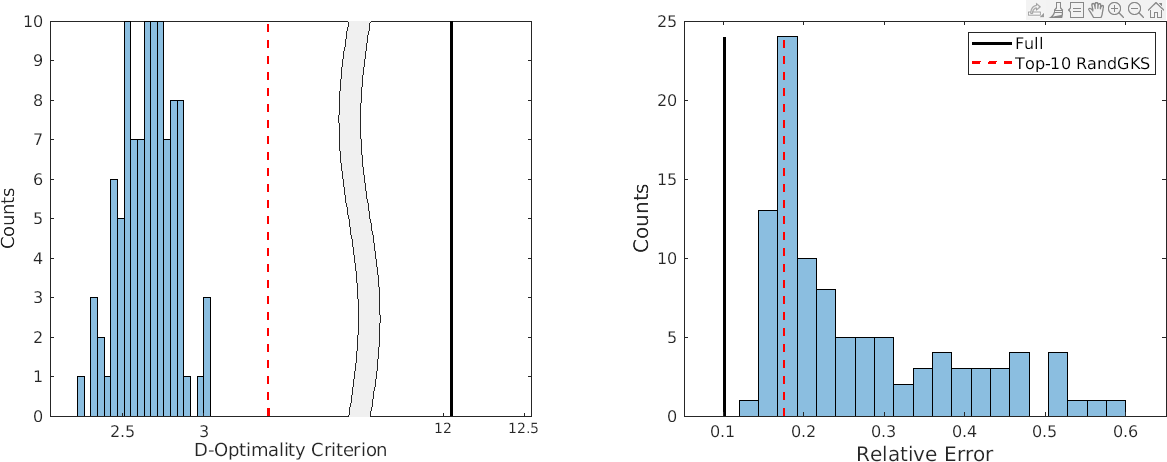}
    \includegraphics[scale=0.35]{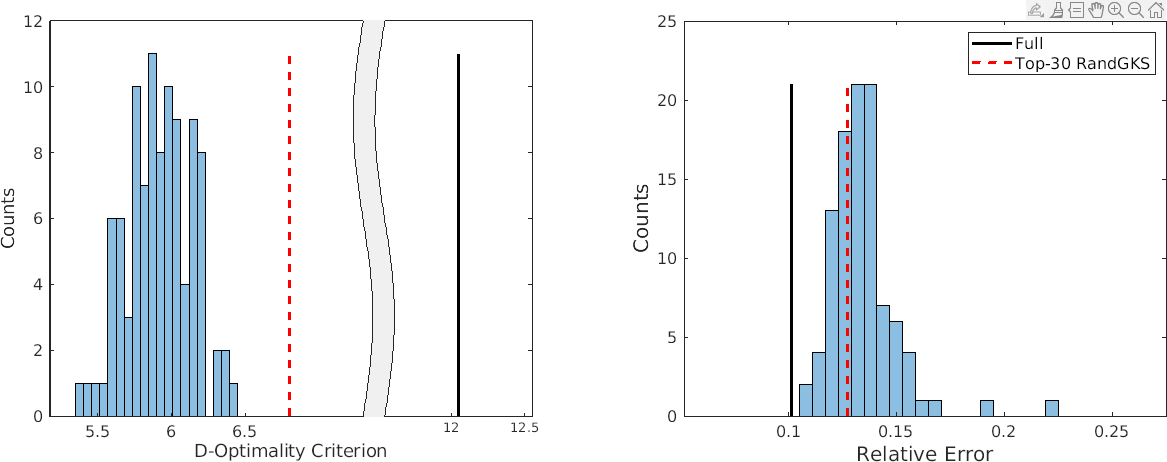}
    \caption{Evaluating the RandGKS algorithm against random sensor placements in the Heat test problem. We include $100$ random sensor selections, with $k = 10$ (top row) and $k=30$ (bottom row), to plot the D-optimality and relative error histograms. RandGKS always has the highest D-optimality and reasonable relative error for all $k$. }
    \label{fig:diffusioncmprand}
\end{figure}

Next, we investigate two specific choices of $k$ ($k=10$ and $k=30$) and compare RandGKS against $100$ random sensor placements in \cref{fig:diffusioncmprand}. 
The results are displayed in \cref{fig:diffusioncmprand} with D-optimality results on the left panel and relative errors on the right; furthermore, the top row corresponds to $k=10$ and the bottom row corresponds to $k=30$. The histogram of the random designs is plotted in blue, along with a solid black line indicating Full and a dashed red line indicating RandGKS. We first consider the D-optimal criterion. We can see that RandGKS beats all the random designs in terms of D-optimality but is far from that of the full operators for both $k=10$ and $k=30$. Next, we consider the relative errors. Here, RandGKS beats most but not all the random designs. This is expected as we are optimizing only for D-optimality and not reconstruction with our algorithms. The discrepancy with the full operator in terms of relative error is not as drastic as that of D-optimality.

\subsection{Comparison of the different methods}\label{sec:algcmp}
In this set of experiments, we compare RandGKS with other methods proposed in this paper and existing methods. We give a brief description of the settings for the algorithms in \Cref{ssec:setup}. 

\begin{table}[!ht]
\small
\centering
\caption{Comparison of the algorithms on Heat ($k=30$) and Seismic ($k=50$).}
\begin{tabular}{c|c|c|c|c|c}
\toprule
Problem               & Algorithm & D-optimality & Relative Error & $\norm{\M{V}_{11}^{-1}}_2$ & Time \\
\midrule
\multirow{5}{*}{Heat} & Full      & 12.0491      & 0.1015         & ---         & ---         \\
                      & RandGKS   & 6.8017       & 0.1274         & 5.8193      & 124.5957 s  \\
                      & Hybrid    & 6.6621       & 0.1194         & 6.6528      & 122.8767 s  \\
                      & Greedy    & 7.0857       & 0.1319         & $\infty$    & 1348.8915 s \\
                      & RAF       & 6.6441       & 0.1263         & 10.413      & 41.6364 s   \\
\midrule
\multirow{5}{*}{Seismic} & Full      & 13.1040      & 0.2934         & ---         &  --- \\
                         & RandGKS   & 5.6962       & 0.2896         & 2.9166      & 1.9811 s \\
                         & Hybrid    & 5.6753       & 0.2875         & 10.583      & 2.1143 s \\
                         & Greedy    & 5.9211       & 0.2956         & 4.7025e+07  & 66.5252 s \\
                         & RAF       & 5.6760       & 0.2930         & 15.636      & 0.6742 s \\
\bottomrule
\end{tabular}
\label{tbl:algcmp}
\end{table}

In \cref{tbl:algcmp}, we display the D-optimal criteria and relative errors for our \ac{oed} methods and the full operator. This table reports results for both Heat and Seismic problems with $k=30$ and $k=50$ sensors, respectively. Note that the results of Full and RandGKS have been previously reported already, but we are including them for comparison. All methods perform similarly in both metrics. We see that for both Heat and Seismic, the Greedy method has the largest D-optimality while Hybrid has the least relative error. The \ac{raf} algorithm's performance is comparable to the other methods, even though it is much cheaper. We also report the value of $\norm{\M{V}_{11}^{-1}}_2$ which appears in the denominator of the lower bound from \cref{thm:aksbounds}. While the Greedy method performed well in our experiments in terms of D-optimality, the value of $\norm{\M{V}_{11}^{-1}}_2$ can be large, suggesting that Greedy may select (nearly) dependent columns.  Finally, we report the Wall clock times taken by the different methods to perform \ac{oed}. \ac{raf} is clearly the fastest method, with the other randomized algorithms performing reasonably. The Greedy method is the most expensive due to the large number of \ac{pde} evaluations involved in its computation. Note that in our implementation, we assumed that the operator $\M{A}$ is not formed explicitly.  Further  details of the computing environment and timing studies are in \cref{ssec:timing}.

\subsection{Data completion}
\begin{figure}[!ht]
    \centering
    \includegraphics[width=0.47\textwidth]{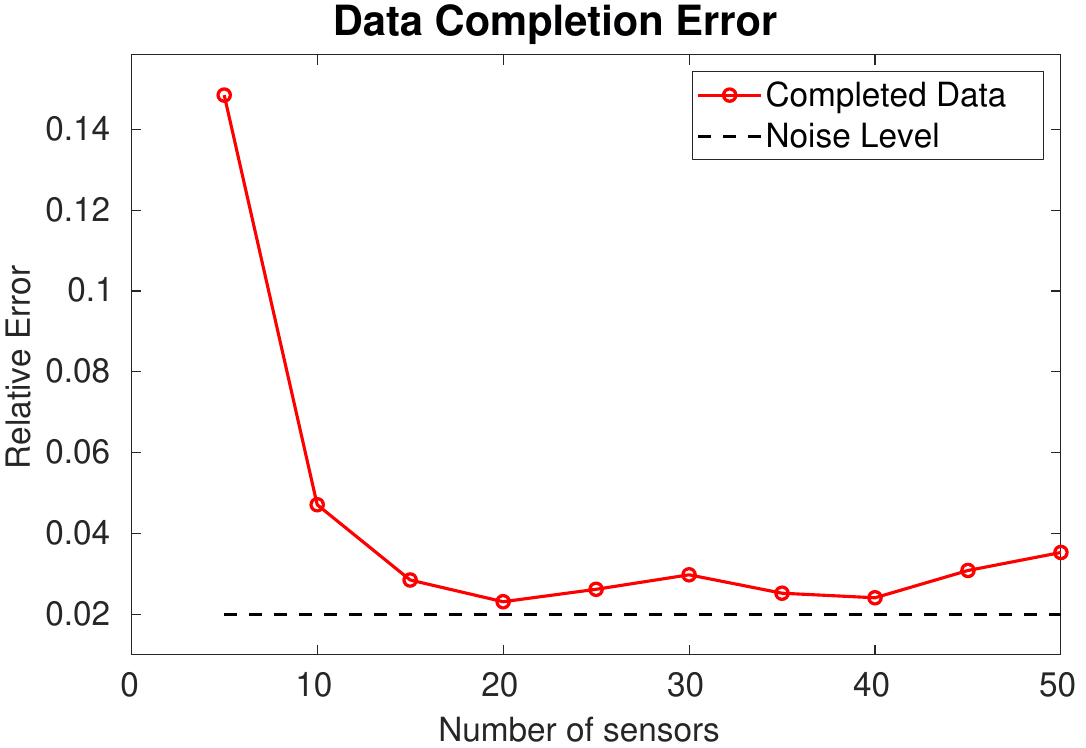}
    \includegraphics[width=0.47\textwidth]{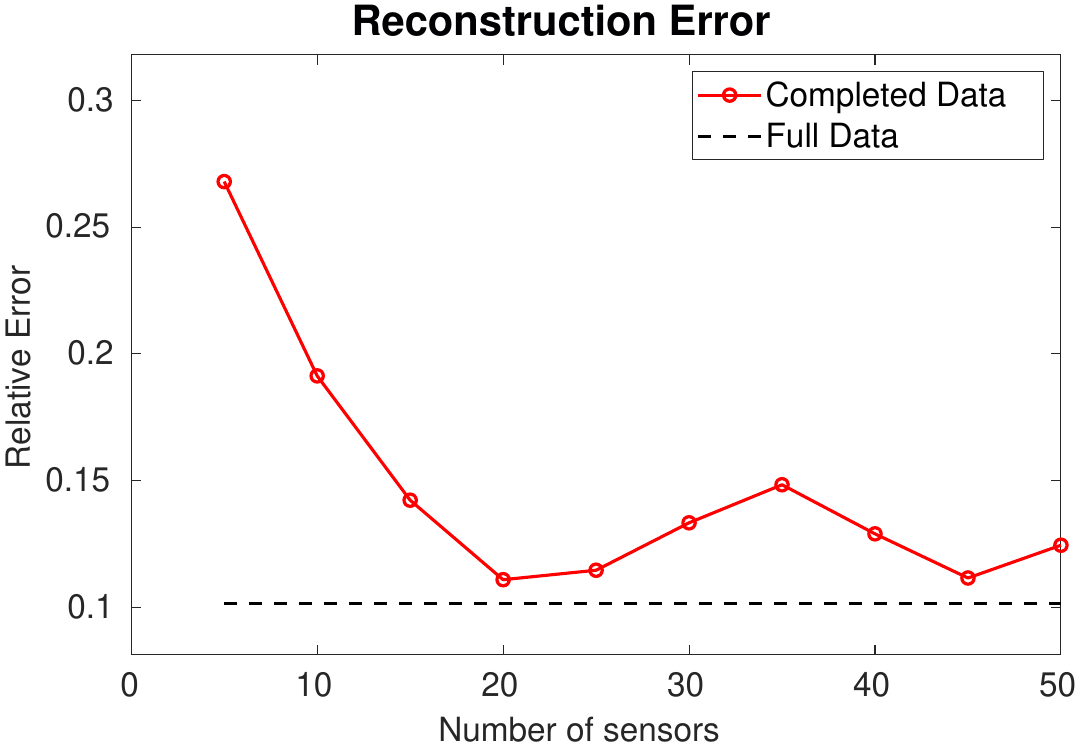}
    \caption{Results from the data completion experiments for the Heat problem. (left) The error in the completed data with increasing number of sensors,  and (right) the reconstruction error using the completed data to compute the MAP estimate. }
    \label{fig:diffusiondatacomp}
\end{figure}
Next, we consider the performance of the \ac{bdeim} approach to complete the data applied to the Heat problem (see \cref{sec:datacomp}).
We increase the number of sensors selected from $5$ to $50$ using the RandGKS algorithm and plot the relative errors in the completed data (left panel) and reconstruction (right panel) in \cref{fig:diffusiondatacomp}.
The relative error for data completion is defined as $\norm{\V{d}-\MB{P}\V{d}}_2/\norm{\V{d}}_2$, where $\V{d}$ is the data from all sensors, and $\MB{P}$ is the \ac{bdeim} projector.
The reconstruction error is computed between the \ac{map} estimate obtained by using $\MB{P}\V{d}$ as data and the true initial conditions.
From the left panel we can see that data completion improves while $k \le 20$ before stagnating at around 2\%.
Recall that 2\% noise has been added to the data, which naturally limits the data recovery.
We see similar trends in the right panel with the \ac{map} estimate.
The reconstruction error improves till around $20$ sensors before stagnating close to the error obtained by using the full operator. 
\begin{figure}[!ht]
    \centering
    \includegraphics[width=0.47\textwidth]{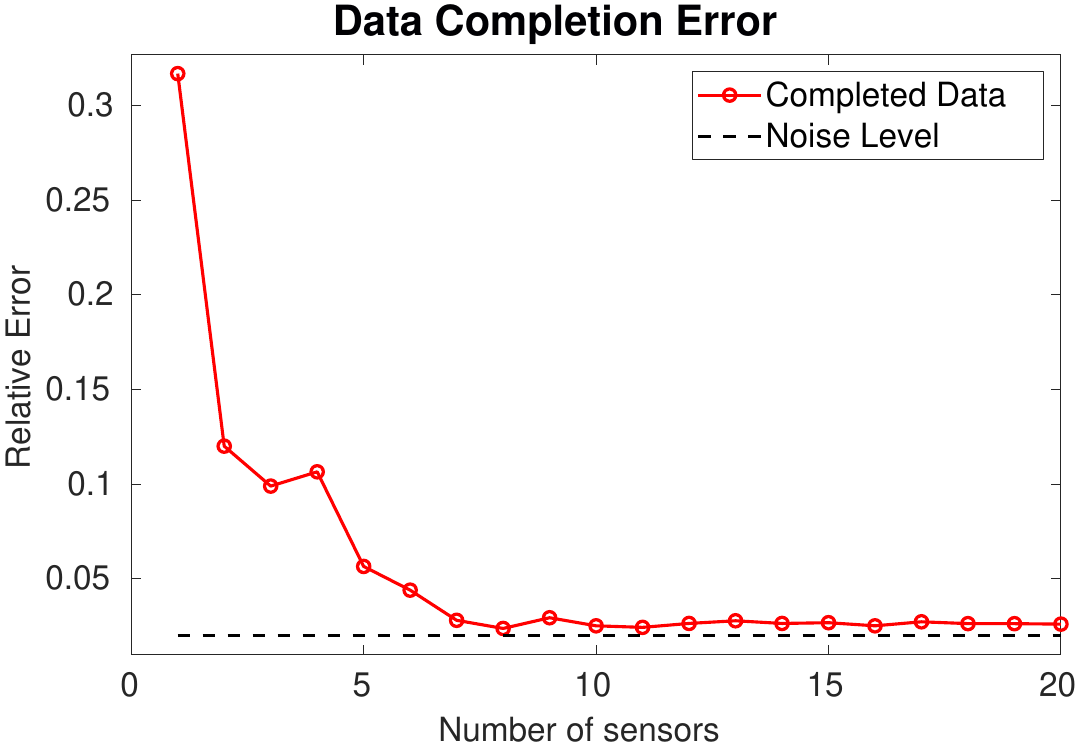}
    \includegraphics[width=0.47\textwidth]{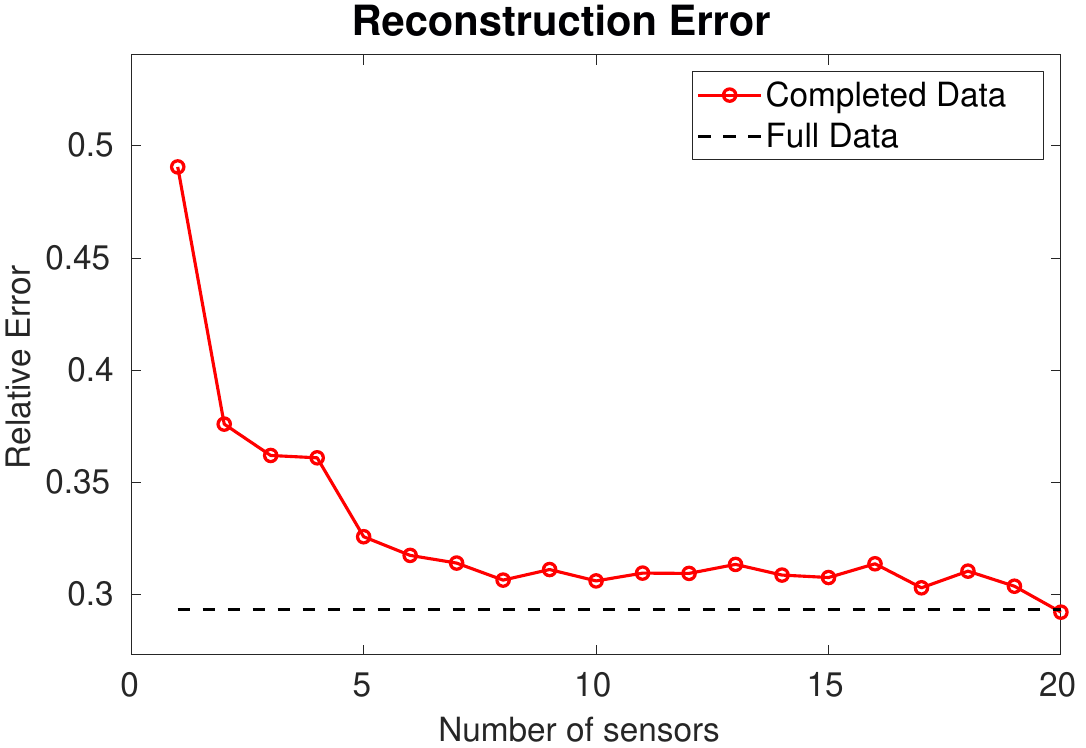}
    \caption{Data completion experiments for Seismic with (left) error in the completed data versus number of sensors and (right) reconstruction error when the completed data is used to find the \ac{map} point. In both cases the error decreases till 8 sensors and then stagnates.}
    \label{fig:seismicdatacomp}
\end{figure}
We also report the performance of \ac{bdeim} on the Seismic problem as we increase the number of sensors used from $1$ to $20$ in \cref{fig:seismicdatacomp}.
We are able to complete the data up to the $2\%$ noise level and recover the \ac{map} estimate with almost the same accuracy as the full operator.

\section{Conclusions}
\label{sec:conc}
This paper presents several algorithms for D-optimal sensor placement in Bayesian inverse problems by leveraging connections to \ac{cssp}. We established several algorithms---both deterministic and randomized---with provable guarantees. The algorithms are computationally efficient, easy to implement, and do not require much tuning. Numerical experiments demonstrate the performance of the methods. There are many avenues for future work. First, can we extend this to other criteria, such as A-optimality and goal-oriented? Second, are these techniques applicable to non-linear inverse problems? Third, can we extend the techniques to correlated noise? A possible way forward is to consider the techniques in~\cite[Section 4.5]{drmac2018discrete}. 
Fourth, based on the connection to \ac{cssp}, can other techniques such as volume-sampling and determinantal point processes~\cite{belhadji2020determinantal} offer tangible benefits?

\section*{Acknowledgments}
We would like to thank Ahmed Attia for his help with PyOED and also thank Alen Alexanderian, Wendy Di, Jayanth Jagalur, and Youssef Marzouk for helpful discussions. 
\appendix
\label{sec:appendix}

\section{Some background results}\label{sec:back}
Let $\M{C},\M{D}\in \R^{n\times n}$ be symmetric matrices. We say that $\M{C} \preceq \M{D}$ (alternatively, $\M{D} \succeq \M{C}$) if $\M{D}-\M{C}$ is positive semidefinite; see~\cite[Chapter 7.7]{horn2013matrix} for additional details. By Weyl's inequality, $\lambda_i(\M{C})\le \lambda_i(\M{D})$ for $1\le i \le n$, where $\lambda_i(\M{C})$ denote the eigenvalues of $\M{C}$ in decreasing order. For short, we also write $\M{C} \succeq \M{0}$ if $\M{C}$ is symmetric positive semidefinite. We also have $\M{XCX}\t \preceq \M{XDX}\t$ where $\M{X} \in \R^{m\times n}$. If $\M{0}\preceq \M{C} \preceq \M{D}$, then $(\M{I}+\M{C})^{-1} \succeq (\M{I}+\M{D})^{-1}$. Similarly, this also implies~\cite[Lemma 9]{alexanderian2018efficient}
\[ 0  \le \logdet(\M{I}+\M{D}) - \logdet(\M{I}+\M{C}) \le \logdet(\M{I} + \M{D}-\M{C}). \]  
We will also need Sylvester's determinant identity
\[ \logdet(\M{I}+\M{AA}\t) = \logdet(\M{I}+\M{A}\t\M{A}). \]
This can be seen for example from the fact that $\M{AA}\t$ and $\M{A}\t\M{A}$ have the same nonzero eigenvalues.

\section{Proof of NP-Hardness}\label{sec:nphard}
In this section, we give a proof of \cref{prop:doptnp}. We restate the decision problem in a slightly different form:
\begin{description}
    \item [Problem] D-optimal sensor placement.
    \item \revised{[Instance] Let $\M{A} \in \R^{M\times N}$ with rank at least $k$, and $\tau \in \R_+$}.
    \item \revised{[Question] Does there exist a matrix $\M{C}\in \R^{M\times k}$ that contains columns from the matrix $\M{A}$, i.e., $\M{C} = \M{A}(:, S)$, such that $\phi(S) \ge \tau$?}
\end{description}

We give a reduction based on the `exact cover by 3-sets' (X3C) problem which is known to be NP-complete. This proof technique is based on Civril and Magdon-Ismail~\cite{civril2009selecting}.
\begin{description}
    \item [Problem] Exact cover by 3-sets (X3C).
    \item [Instance] A set $\mc{S} = \{1,\dots,3m\}$ and a collection of subsets $\mc{K} = \{\mc{K}_1,\dots,\mc{K}_n\}$ each with cardinality $3$.  
    \item [Question]  Is there a subset $\mc{K}'\subset \mc{K}$ that forms an exact cover for $\mc{S}$ (Exact cover means every element of $\mc{S}$ appears exactly once in $\mc{K}'$.)?
    \end{description}

Before we launch into the proof of \cref{prop:doptnp}, we need the following lemma. 
\begin{lemma}\label{lem:orth} Let $\M{Z}\in \R^{M \times k}$ be a matrix with $k \le M$ and with column norms $1$. Then $\logdet(\M{I}+\M{ZZ}\t) \le k\log 2$ with equality if and only if columns of $\M{Z}$ are orthonormal.
\end{lemma}
\begin{proof}
    Let us denote the columns of $\M{Z} = \bmat{\V{z}_1 & \dots & \V{z}_k}$ and  singular values of $\M{Z}$ as $\sigma_1 \ge \dots \sigma_k \ge 0$.  Then  $\|\M{Z}\|_F^2 = \sum_{j=1}^k\|\V{z}_j\|_2^2 = \sum_{j=1}^k\sigma_j^2  = {k}$ and 
    \[ \logdet(\M{I}+\M{ZZ}\t) = \logdet(\M{I}+\M{Z}\t \M{Z}) \le \sum_{j=1}^k\log(1+\|\V{z}_j\|_2^2) =  k\log 2, \]
    where in the first step we used Sylvester's determinant identity and in the second step we used Hadamard's determinant inequality~\cite[Theorem 7.8.1]{horn2013matrix}. This establishes the inequality. We now consider the case of the equality. If $\M{Z}$ has orthonormal columns, all its singular values are $1$, so equality is achieved. To show the converse, consider the equality
    \[ \logdet(\M{I}+\M{ZZ}\t) =\sum_{j=1}^k\log(1+\sigma_j^2) = k \log 2,\]
    which we can rewrite as $\prod_{j=1}^k(1+\sigma_j^2) = 2^k$ or $\sum_{j=1}^k \log ( (\sigma_j^2+1)/2) = 0$. Together with $\sum_{j=1}^k\sigma_j^2  = {k}$, we have 
    \[ \sum_{j=1}^k \left[\frac{\sigma_j^2 - 1}{2} - \log ( (\sigma_j^2+1)/2) \right] = 0. \]
    Consider the function $g(x) = (x-1)/2 - \log((x+1)/2)$. This function is nonnegative and it has a unique global minimizer at $x = 1$ and $g(1)= 0$. This implies that the only solution to $\sum_{j=1}^k g(\sigma_j^2) = 0$ is $\sigma_1 = \dots = \sigma_k= 1$, or $\M{Z}$ has orthonormal columns. 
\end{proof}

We are ready to prove \cref{prop:doptnp}. 
\revised{We show that every instance of X3C is true if and only if the corresponding decision problem for D-optimal sensor placement is true.}
To this end, we construct the matrix $\M{A}\in \R^{3m \times n}$  with entries ($1 \le i \le 3m$, $1\le j \le n$)
\[a_{ij} = \left\{\begin{array}{ll}\frac{1}{\sqrt{3}} & i \in \mc{K}_j  \\ 0 & \text{otherwise}.\end{array} \right.\]
Thus, in this instance of X3C, $M = 3m$, $N=n$, \revised{$k=m$, and $\tau = m \log 2$}. We note that the columns of $\M{A}$ have norm $1$ so \cref{lem:orth} applies. 

From the proof of~\cite[Theorem 4]{civril2009selecting}, the instance of X3C is true if and only there exists a matrix $\hat{\M{C}}\in \R^{3m\times m}$ (containing columns from $\M{A}$ \revised{indexed by set $\hat{S}$}) with orthonormal columns.  
\revised{
By \cref{lem:orth}, this is equivalent to showing that $\phi({\hat{S}}) = \tau = m \log 2$.
Therefore, the instance of X3C is true if and only if there is a set $\hat{S}$ with $\phi(\hat{S}) = m \log 2$.  
}
This completes the proof.

\section{Alternatives}\label{sec:alternatives}
In \Cref{ssec:altgks}, we give an alternative to the \ac{gks} bounds proposed in \Cref{ssec:oed_cssp}, and in \Cref{ssec:maxvol} we give some alternative techniques using maximum-volume. In \Cref{ssec:proofdetails}, we give additional details of the proof of \cref{thm:randoed}.
\subsection{Alternatives to the GKS bounds}\label{ssec:altgks} 

\Cref{ssec:oed_cssp,ssec:struct} show two different lower bounds for 
\revised{$\phi(S)$}.
The first is the \ac{gks} bounds which can be extended to the top-$k$ singular values as follows.
Let the truncated \ac{svd} of $\M{A}$ be $\M{U}_k\M{\Sigma}_k\M{V}_k^\Tra$ with $\M{V}_k \in \Rn{m \times k}$.
Let $\M{V}_k$ be partitioned exactly as \cref{eqn:vkpart} into
$$
\M{V}_k^\Tra\bmat{\M{\Pi}_1 & \M{\Pi}_2} = \bmat{\M{V}_{11} & \M{V}_{22}}.
$$
Now let us denote the $\ell\times\ell$ principal submatrix of $\M{V}_{11}$ as $\M{V}_{11}(1:\ell,1:\ell) = \M{V}_{(\ell,\ell)}$.
Then by the \ac{gks} bounds we have~\cite[Theorem 5.5.2]{golub2012matrix},
\begin{equation}
\frac{\sigma_j\pr{\M{A}}}{\norm{(\M{V}_{(j,j)})^{-1}}_2} \le \sigma_j\pr{\M{C}} \le \sigma_j\pr{\M{A}}, \qquad 1 \le j \le k.
\label{eqn:gvlbounds}
\end{equation}
Notice that the block matrix varies for each singular value making it unwieldy to apply for the D-optimal criterion as in \cref{cor:srrqraksbounds}.
Contrast this with the bounds in \cref{thm:aksbounds} which gives us
\begin{equation}
\frac{\sigma_j\pr{\M{A}}}{\norm{\M{V}_{11}^{-1}}_2} \le \sigma_j\pr{\M{C}} \le \sigma_j\pr{\M{A}}, \qquad 1 \le j \le k.
\label{eqn:aksbounds}
\end{equation}
In situations where a tighter lower bound for 
$\revised{\phi(S)}$ is desired, one can compare \cref{eqn:gvlbounds,eqn:aksbounds} for every singular value.
Formally, let define $$d_j = 1/\min\pr{\norm{\M{V}_{(j,j)}^{-1}}_2, \norm{\M{V}_{11}^{-1}}_2} \qquad 1 \le j \le k$$ and let $\M{D} = \diag{d_1,\dots,d_k} \in \Rn{k \times k}$.
Then using similar arguments as those in \Cref{ssec:csspproofs} we get
$$
\revised{
\Psi(\M{\Sigma}_k\M{D}) \le \phi(S) \le \phi(S^{\rm opt}) \le \Psi(\M{\Sigma}_k) \le \phi(\br{m}).
}
$$

\subsection{Maximum volume-based approaches}\label{ssec:maxvol} The connection to maximum volume suggests some alternative approaches. That is, we should select columns of $\Mh{A} \equiv \bmat{ \M{I} & \M{A}\t}\t$, rather than directly select columns of $\M{A}$. First notice that the singular values of $\Mh{A}$ are related to $\M{A}$ as $$\sigma_i(\Mh{A}) = \sqrt{1 + \sigma_i(\M{A})^2}\quad \text{for} \quad 1\le i \le \min\{m,n\}. $$ 
If we apply \ac{srrqr} with parameter $f \ge 1$ to $\Mh{A}$ then we get a matrix $\Mh{C}$ with $k$ columns from $\Mh{C}$ of the form   
\[ \Mh{C} = \bmat{ * \\ \M{C}} \in \R^{(n+m)\times k},  \]
where $\M{C}$ contains $k$ columns from $\M{A}$. The block denoted by $*$ is zero except for $k$ rows which contain the $k\times k$ identity matrix. From the \ac{srrqr} bounds 
\[ \frac{\sigma_i(\Mh{A}) }{q_f(m,k)} \le \sigma_i( \Mh{C} )\le \sigma_i(\Mh{A}) \qquad 1 \le i \le k.  \]
This leads to the bounds 
\revised{
$$ \Psi(\M\Sigma_k) - 2k\log q_f(m,k) \le \phi(S) \le \Psi(\M\Sigma_k). $$
}
That is, we get an additive bound rather than a multiplicative bound. Since the factor $q_f(m,k) \ge 1$, it is easy to see that 
\revised{$\Psi(\M\Sigma_k/q_f(m,k)) \ge \Psi(\M\Sigma_k) - 2k\log q_f(m,k)$.}
Therefore, this approach produces a looser upper bound, so we did not investigate this further. We leave it to future work to find approaches that exploit this connection to yield better results.

\subsection{Details of the proof of \cref{thm:randoed}}\label{ssec:proofdetails}
We want to provide details of the statement $\|\M\Omega\|_2 \|(\M{U}_k\M\Omega)^\dagger\|_2 \le C_g $ with probability at least $1-\delta$. From~\cite[Proposition 10.2]{halko2011finding}, $\mathbb{E}[\|\M\Omega\|_2] \le \sqrt{\frac{n}{d}} + 1$. Next, $\|\cdot\|_2$ is a Lipschitz function with Lipschitz constant $1$, so by the previous result and~\cite[Proposition 10.3]{halko2011finding} 
\[\mathbb{P}\left\{\|\M\Omega\|_2 \ge \sqrt{\frac{n}{d}} + 1  + t  \right\} \le e^{-t^2/2}.  \]
Set $e^{-t^2/2} = \delta/2$ to obtain $t =\sqrt{2\log(2/\delta)}$. Therefore, $\|\M\Omega\|_2 \le  \sqrt{\frac{n}{d}} + 1  +\sqrt{2\log(2/\delta)}$ with probability at least $1-\delta/2$.

Next, we tackle $(\M\Omega \M{U}_k)$, which we can write as $\M{G}\M{U}_k/\sqrt{d}$ where $\M{G} \in \R^{d\times n}$ is a standard Gaussian matrix (i.i.d. entries from $\mathcal{N}(0,1)$). By the rotational invariance, $\M{G}\M{U}_k \in \R^{d\times k}$ is also a standard Gaussian random matrix. Therefore, by~\cite[Proposition 10.4]
{halko2011finding} 
\[ \mathbb{P}\left\{\|(\M\Omega\M{U}_k)^\dagger\|_2 \ge {\sqrt{d}} \cdot \frac{e\sqrt{d}}{p+1}\cdot t  \right\} \le t^{-(p+1)}.\]
Once again, set $t^{-(p+1)} = \delta/2$, to get $t = (2/\delta)^{1/(p+1)}$. Therefore, $\|(\M\Omega\M{U}_k)^\dagger\|_2 \le  {\sqrt{d}} \cdot \frac{e\sqrt{d}}{p+1}\cdot(2/\delta)^{1/(p+1)} $ with probability at least $1-\delta/2$. The proof is completed by a simple union bound argument.

\section{Additional numerical experiments}\label{sec:addexp}
Additional experiments on the same test problems from \Cref{sec:exp}.

\subsection{Summary of algorithmic choices}\label{ssec:setup} We summarize the choices of problem settings that is used in the comparison of algorithms below.
\begin{enumerate}
    \item \textbf{RandGKS}: We implemented \cref{alg:detcssp} where a randomized SVD algorithm (with $q=1$ subspace iterations and oversampling $p=20$) is used to compute the right singular vectors of $\M{A}$ and QR with column pivoting (QRCP) is used for computing the pivot indices that determine the sensor placement.
    \item \textbf{Hybrid}: This is the randomized point selection method described in \cref{alg:leverage_sampling} and then QRCP is used in the second stage of the method to cut down the sampled sensors to exactly the requested number of sensors. The randomized SVD was run with the same parameters as in RandGKS. We use $s = \min\{ \lceil k \log k\rceil, m\}$, the sampling probabilities $\{\pi_j^\beta\}_{j=1}^m$, where $\pi_j^m = \beta \frac{\tau_j}{k} +(1-\beta)\frac{1}{m}$ for $1 \le j \le m$ and $\beta=0.9$ in the numerical experiments.
    \item \textbf{Greedy}: In this approach, we greedily pick columns of $\M{A}$ to maximize the D-optimal criterion. 
    Given indices $I = \{i_1,\dots,i_t\}$ at step $t$, the method selects an index $i_{t+1}$ from $\{1,\dots,m\} \setminus I $ such that the log determinant of $\M{I}+\M{XX}\t$ is maximized, where $\M{X} = \bmat{\V{a}_{i_1} & \dots & \V{a}_{i_{t+1}}}. $ 
    To compute a column norm, we require one matvec with $\M{A}$ to extract the column. 
    Furthermore, this method is expensive compared to the other approaches for large $m$, as it requires $O(mk)$ matvecs with $\M{A}$. 
    \item \textbf{RAF}: This is \cref{alg:randoed} where \ac{qrcp} is used for computing the pivot indices directly from the sketch $\M{Y} = \M{\Omega}\M{A}$. We set $p = 20$ as the oversampling parameter.
\end{enumerate}

\subsection{Additional results for Seismic Tomography}We repeat the experiments from \cref{sec:exp} for the Seismic tomography problem.
The results largely mimic the findings from the Heat problem.
\Cref{fig:seismicvaryk} shows the performance of the RandGKS algorithm as we increase $k$. 
We see diminishing returns in terms of relative errors after 30 sensors, but improving D-optimality (reduction in uncertainty) as we add more sensors.
Next we show the performance of RandGKS versus 100 random sensor selections for $k = 10$ and $k = 50$ in \cref{fig:seismiccmprand}.
RandGKS performs extremely well against the random selections, beating them all in D-optimality and doing extremely well in reconstruction errors as well.

\begin{figure}[!ht]
    \centering
    \includegraphics[width=0.5\textwidth]{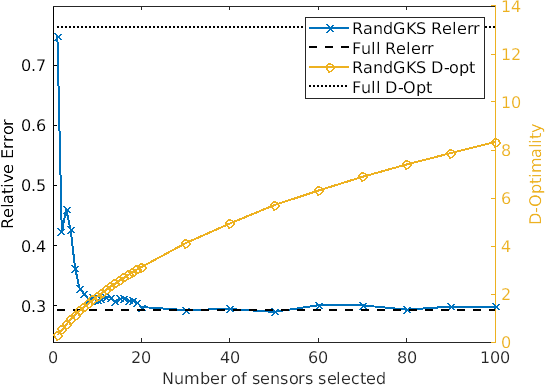}
    \caption{D-optimal criterion and relative error in the RandGKS algorithm with increasing number of sensors. Relative error stagnates around 30 sensors while D-optimality keeps increasing with more sensors.}
    \label{fig:seismicvaryk}
\end{figure}

\begin{figure}[!ht]
    \centering
    \includegraphics[scale=0.38]{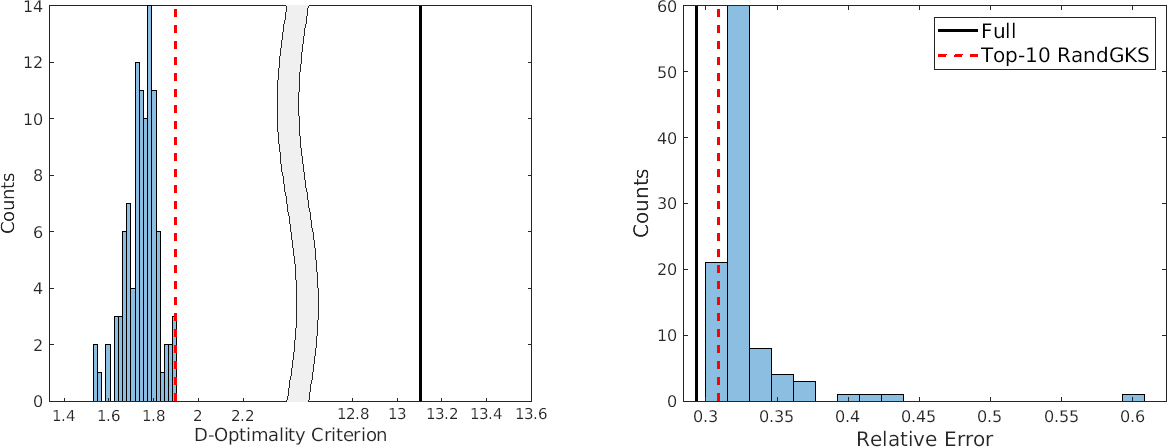}
    \includegraphics[scale=0.38]{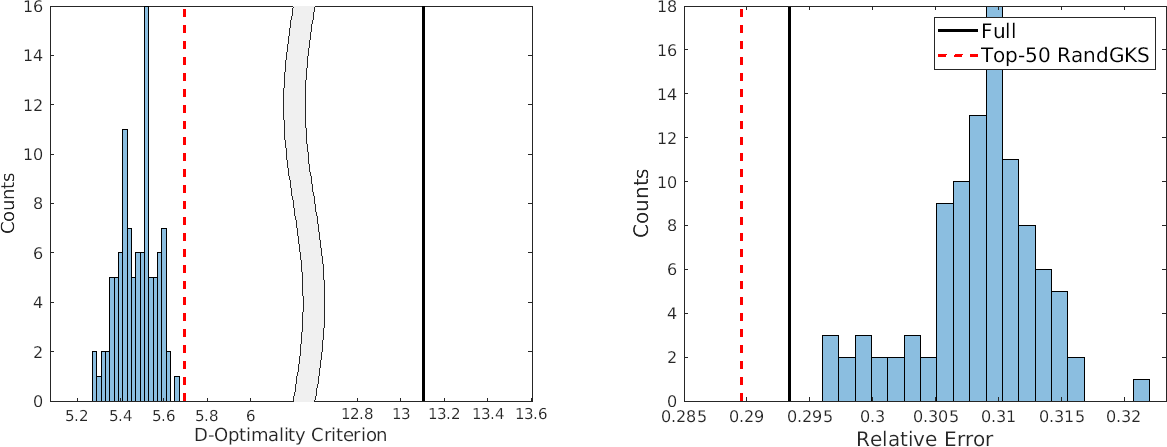}
    \caption{Evaluating the RandGKS algorithm against random sensor placements in the Seismic problem. 100 random sensor selections, with $k=10$ or $k=50$, are used to plot the D-optimality and relative error histograms. RandGKS always has a high D-optimality and reasonable relative error for all $k$.}
    \label{fig:seismiccmprand}
\end{figure}

\subsection{Additional comparison of algorithms}
We expand \cref{tbl:algcmp} for more values of $k$ for both test problems in \cref{tbl:diffusionalgcmp,tbl:seismicalgcmp}.
In addition to D-optimality and relative error, we also report $\norm{\M{V}_{11}^{-1}}_2$ for both test problems.
Recall from \cref{thm:aksbounds}, that this norm appears in the lower bound for the D-optimal criterion. Greedy has the highest D-optimality but other methods produce comparable results. As is seen from the timing results, Greedy is considerably more expensive. 

\begin{table}[!ht]
\centering
\caption{Performance of the different algorithms on the Heat test problem.}
\begin{tabular}{c|c|c|c|c}
\toprule
$k$                 & Algorithm & D-optimality & Relative Error & $\norm{\M{V}_{11}^{-1}}_2$ \\
\midrule
 ---                & Full      & 12.0491      & 0.1015         & ---          \\
\midrule
\multirow{4}{*}{5} & RandGKS & 2.0815 & 0.2144 & 6.0056 \\
                   & Hybrid  & 1.6742 & 0.4214 & 8.3106 \\
                   & Greedy  & 2.1253 & 0.2365 & 18.281 \\
                   & RAF     & 1.9467 & 0.3129 & 23.73  \\
\midrule
\multirow{4}{*}{10} & RandGKS & 3.3966 & 0.1758 & 7.9548   \\
                    & Hybrid  & 3.2382 & 0.2212 & 10.084   \\
                    & Greedy  & 3.5061 & 0.1768 & $\infty$ \\
                    & RAF     & 3.3946 & 0.1721 & 54.074   \\
\midrule
\multirow{4}{*}{20} & RandGKS & 5.2074 & 0.1307 & 5.6349   \\
                    & Hybrid  & 5.1778 & 0.1283 & 5.7384   \\
                    & Greedy  & 5.5779 & 0.1502 & $\infty$ \\
                    & RAF     & 5.2897 & 0.1268 & 11.251   \\
\midrule
\multirow{4}{*}{40} & RandGKS & 7.9544 & 0.1204 & 6.3169   \\
                    & Hybrid  & 7.8053 & 0.1131 & 15.893   \\
                    & Greedy  & 8.3063 & 0.1298 & $\infty$ \\
                    & RAF     & 7.7388 & 0.1075 & 97.507   \\
\midrule
\multirow{4}{*}{50} & RandGKS & 8.8605 & 0.1121 & 7.2188   \\
                    & Hybrid  & 8.7376 & 0.1154 & 23.067   \\
                    & Greedy  & 9.2042 & 0.1229 & $\infty$ \\
                    & RAF     & 8.9459 & 0.1118 & 1166.3   \\
\bottomrule
\end{tabular}
\label{tbl:diffusionalgcmp}
\end{table}

\begin{table}[!ht]
\centering
\caption{Performance of the different algorithms on the Seismic test problem.}
\begin{tabular}{c|c|c|c|c}
\toprule
$k$                 & Algorithm & D-optimality & Relative Error & $\norm{\M{V}_{11}^{-1}}_2$ \\
\midrule
 ---                & Full      & 13.1040      & 0.2934         & ---          \\
\midrule
\multirow{4}{*}{10} & RandGKS & 1.8969 & 0.3089 & 5.8451     \\
                    & Hybrid  & 1.8398 & 0.3092 & 14.004     \\
                    & Greedy  & 2.0442 & 0.4022 & 2.0046e+05 \\
                    & RAF     & 1.8735 & 0.3097 & 12.644     \\
\midrule
\multirow{4}{*}{20} & RandGKS & 3.1218 & 0.2977 & 4.2127     \\
                    & Hybrid  & 3.0885 & 0.3007 & 8.6842     \\
                    & Greedy  & 3.3450 & 0.3227 & 1.8129e+07 \\
                    & RAF     & 3.1315 & 0.3042 & 20.929     \\
\midrule
\multirow{4}{*}{40} & RandGKS & 4.9363 & 0.2945 & 3.1561     \\
                    & Hybrid  & 4.9381 & 0.3034 & 4.6546     \\
                    & Greedy  & 5.1947 & 0.2952 & 2.7784e+07 \\
                    & RAF     & 4.9676 & 0.2992 & 88.5       \\
\midrule
\multirow{4}{*}{60} & RandGKS & 6.3059 & 0.3006 & 2.9216     \\
                    & Hybrid  & 6.3087 & 0.3022 & 5.5387     \\
                    & Greedy  & 6.5605 & 0.2958 & 4.6921e+09 \\
                    & RAF     & 6.3061 & 0.2961 & 13.052     \\
\midrule
\multirow{4}{*}{80} & RandGKS & 7.3967 & 0.2933 & 7.0063     \\
                    & Hybrid  & 7.3979 & 0.2942 & 7.9496     \\
                    & Greedy  & 7.6556 & 0.2926 & 3.2542e+09 \\
                    & RAF     & 7.4074 & 0.2929 & 558.23     \\
\bottomrule
\end{tabular}
\label{tbl:seismicalgcmp}
\end{table}

\subsection{Additional run time experiments}\label{ssec:timing}
We compare the running times of the different algorithms for both the test problems. All our experiments were conducted on a laptop computer with a 12th Gen Intel$^{\text{\textregistered}}$
 Core(TM) i7-1280P processor and 32 GB of memory using Matlab (version R2023b). We report the wall clock time for different operations in \cref{tbl:diffusiontimes,tbl:seismictimes}. These timings were collected over 10 different runs and the median time and standard deviation are reported.

\begin{table}[!ht]
\centering
\caption{Timing results of the different algorithms on the Heat test problem.}
\begin{tabular}{c|c|c|c|c}
\toprule
k                   & Algorithm        & \begin{tabular}[c]{@{}c@{}}Matrix-free time\\ median ($\pm$ std)\end{tabular} & PDE solves & \begin{tabular}[c]{@{}c@{}}Dense time\\ median ($\pm$ std)\end{tabular} \\
\midrule
---                 & Forward Op    & 0.4341 ($\pm$ 0.0146)     & ---  & 3.90e-05 ($\pm$ 4.95e-05) \\
---                 & Adjoint Op    & 0.4458 ($\pm$ 0.0086)     & ---  & 3.85e-05 ($\pm$ 4.09e-05) \\
\midrule
\multirow{4}{*}{10} & RandGKS       & 70.1244 ($\pm$ 2.2287)    & 120   & 0.0053 ($\pm$ 0.0023)     \\
                    & Hybrid        & 68.9826 ($\pm$ 6.5980)    & 120   & 0.0058 ($\pm$ 0.0017)     \\
                    & Greedy        & 510.9801 ($\pm$ 4.1313)   & 955  & 0.1472 ($\pm$ 0.0301)     \\
                    & RAF           & 21.5377 ($\pm$ 1.0847)    & 30   & 0.0034 ($\pm$ 0.0010)     \\
\midrule
\multirow{4}{*}{30} & RandGKS       & 124.5957 ($\pm$ 1.5266)   & 200  & 0.0084 ($\pm$ 0.0009)     \\
                    & Hybrid        & 122.876 ($\pm$ 1.1271)    & 100  & 0.0078 ($\pm$ 0.0016)     \\
                    & Greedy        & 1348.8915 ($\pm$ 21.7416) & 2565 & 0.2843 ($\pm$ 0.0378)     \\
                    & RAF           & 41.6364 ($\pm$ 0.2008)    & 50   & 0.0034 ($\pm$ 0.0004)     \\
\bottomrule
\end{tabular}
\label{tbl:diffusiontimes}
\end{table}

\begin{table}[!ht]
\caption{Timing results of the different algorithms on the Seismic test problem.}
\begin{tabular}{c|c|c|c|c}
\toprule
k                   & Algorithm  & \begin{tabular}[c]{@{}c@{}}Matrix-free time\\ median (std)\end{tabular} & PDE solves & \begin{tabular}[c]{@{}c@{}}Dense time\\ median (std)\end{tabular} \\
\midrule
---                 & Forward Op & 0.0056 ($\pm$ 0.0009)  & ---        & 7.45e-05 ($\pm$ 1.03e-04)  \\
---                 & Adjoint Op & 0.0048 ($\pm$ 0.0004)  & ---        & 3.85e-05 ($\pm$ 4.09e-05)  \\
\midrule
\multirow{4}{*}{10} & RandGKS    & 0.6245 ($\pm$ 0.0320)  & 120         & 0.0078 ($\pm$ 0.0009)      \\
                    & Hybrid     & 0.6194 ($\pm$ 0.0227)  & 120         & 0.0086 ($\pm$ 0.0005)      \\
                    & Greedy     & 14.1190 ($\pm$0.5570)  & 2515       & 0.2603 ($\pm$0.0078)       \\
                    & RAF        & 0.2191 ($\pm$ 0.0049)  & 30         & 0.0029 ($\pm$ 0.0003)      \\
\midrule
\multirow{4}{*}{50} & RandGKS    & 1.9811 ($\pm$ 0.1418)  & 280        & 0.0195 ($\pm$ 0.0027)      \\
                    & Hybrid     & 2.1143 ($\pm$ 0.3147)  & 280        & 0.0205 ($\pm$ 0.0037)      \\
                    & Greedy     & 66.5252 ($\pm$ 3.0533) & 11575      & 2.4353 ($\pm$ 1.0554)      \\
                    & RAF        & 0.6742 ($\pm$ 0.0098)  & 70         & 0.0075 ($\pm$ 0.0004)      \\
\bottomrule
\end{tabular}
\label{tbl:seismictimes}
\end{table}

\begin{figure}[!ht]
    \centering
    \includegraphics[width=0.47\textwidth]{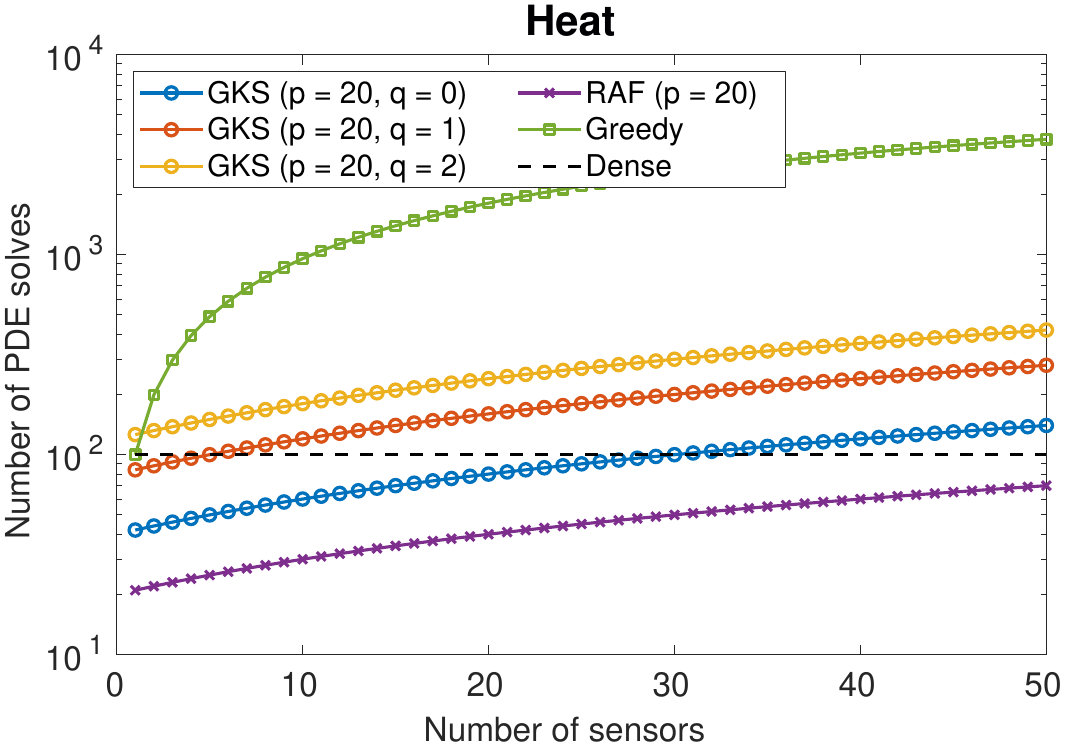}
    \includegraphics[width=0.47\textwidth]{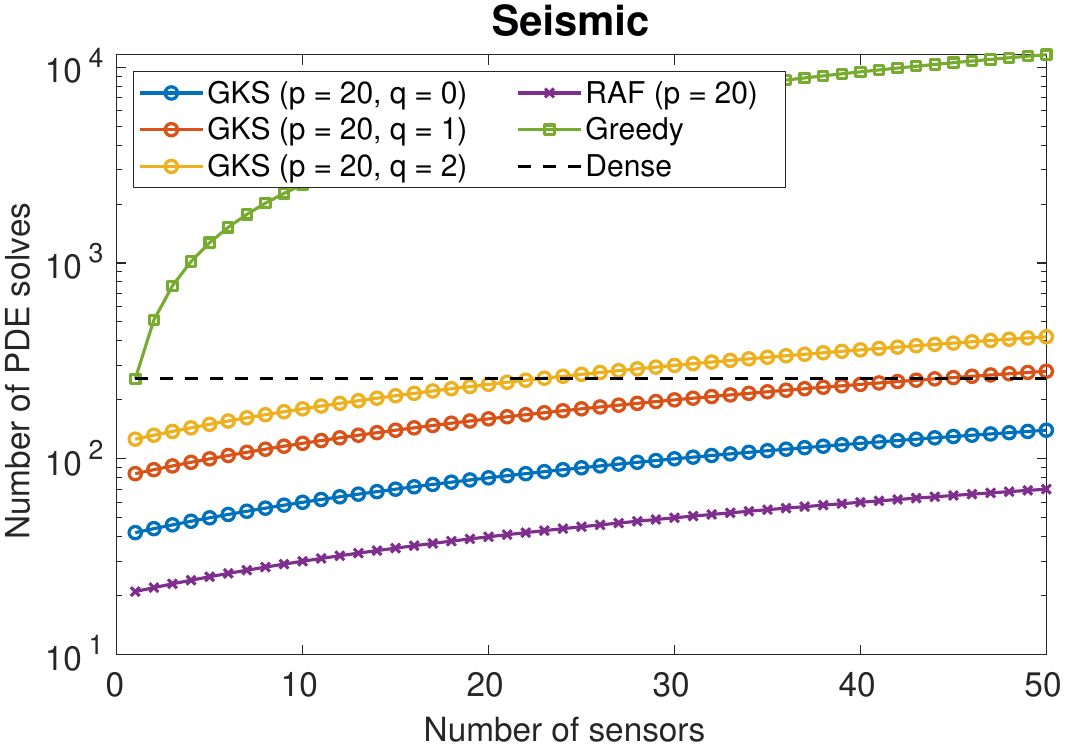}
    \caption{Scaling the number of PDE solves with the number of sensors selected. Notice that the Greedy algorithm performs a lot more PDE solves, up to two orders of magnitude more, than the randomized methods.}
    \label{fig:matvecsvsk}
\end{figure}

From \cref{tbl:diffusiontimes,tbl:seismictimes} it clear that Randomized Adjoint-free OED (RAF-OED) is the fastest method among the four algorithms tested. RandGKS and Hybrid take the next two spots with the Greedy method lagging significantly behind the three randomized methods. The reason for this can be seen from the number of PDE solves incurred by the Greedy method which is an order of magnitude larger than the other methods. 

\Cref{fig:matvecsvsk} shows the growth in the number of PDE solves for the different methods as we increase the number of sensors selected $k$ for our test problems. The number of candidate locations, $m$, is 100 and 256 for the Heat and Seismic problems respectively. 
The randomized SVD requires $(2q + 2)(k + p)$ matvecs for a given oversampling parameter $p$ and subspace iterations $q$.
The Greedy method performs $mk - k(k+1)/2$ solves for a given $k$.
Finally, RAF-OED needs only $k+p$ solves.
Notice that only the Greedy method grows as $O(mk)$ whereas all the other methods require only $O(k)$ solves.
From \cref{tbl:diffusiontimes,tbl:seismictimes} and \cref{fig:matvecsvsk}, it is evident that Greedy method is substantially slower.

Since our problems are of modest size, we can simply ``densify'' the matrices using $m$ PDE solves. This is not possible for larger problems due to the prohibitive memory costs associated with storing the entire forward operator. Even in this setting, we can see from \cref{tbl:diffusiontimes,tbl:seismictimes} that our proposed methods outperform the Greedy method. From \cref{fig:matvecsvsk}, we can see that the matrix-free versions of the randomized algorithms requires fewer solves than the dense greedy method for small values of $k$.

\section{Software availability}
\label{sec:codeloc}
A software package containing all the scripts and instructions needed to reproduce the numerical experiments in \Cref{sec:exp,sec:addexp} can be found in the Github repository\footnote{\url{https://github.com/RandomizedOED/css4oed}}.
Both Matlab and Python scripts are provided which are compatible with the AIR Tools II/IR Tools~\cite{hansen2018air,gazzola2019ir} and PyOED~\cite{chowdhary2023pyoed} packages respectively.

\bibliographystyle{siamplain}
\bibliography{references}

\begin{thebibliography}{10}

\bibitem{alexanderian2021optimal}
{\sc A.~Alexanderian}, {\em Optimal experimental design for
  infinite-dimensional {B}ayesian inverse problems governed by {PDE}s: A
  review}, Inverse Probl., 37 (2021), p.~043001.

\bibitem{alexanderian2018efficient}
{\sc A.~Alexanderian and A.~K. Saibaba}, {\em Efficient {D}-optimal design of
  experiments for infinite-dimensional {B}ayesian linear inverse problems},
  SIAM J. Sci. Comput, 40 (2018), pp.~A2956--A2985.

\bibitem{attia2022stochastic}
{\sc A.~Attia, S.~Leyffer, and T.~S. Munson}, {\em Stochastic learning approach
  for binary optimization: Application to {B}ayesian optimal design of
  experiments}, SIAM J. Sci. Comput, 44 (2022), pp.~B395--B427.

\bibitem{avron2013faster}
{\sc H.~Avron and C.~Boutsidis}, {\em Faster subset selection for matrices and
  applications}, SIAM J. Matrix Anal. Appl., 34 (2013), pp.~1464--1499.

\bibitem{belhadji2020determinantal}
{\sc A.~Belhadji, R.~Bardenet, and P.~Chainais}, {\em A determinantal point
  process for column subset selection}, J. Mach. Learn. Res., 21 (2020),
  pp.~8083--8144.

\bibitem{bhatia2013matrix}
{\sc R.~Bhatia}, {\em Matrix analysis}, vol.~169, Springer Science \& Business
  Media, 2013.

\bibitem{boutsidis2009unsupervised}
{\sc C.~Boutsidis, P.~Drineas, and M.~W. Mahoney}, {\em Unsupervised feature
  selection for the $ k $-means clustering problem}, Advances in neural
  information processing systems, 22 (2009),
  \url{https://doi.org/10.5555/2984093.2984111}.

\bibitem{boutsidis2009improved}
{\sc C.~Boutsidis, M.~W. Mahoney, and P.~Drineas}, {\em An improved
  approximation algorithm for the column subset selection problem}, in
  Proceedings of the twentieth annual ACM-SIAM symposium on Discrete
  algorithms, SIAM, 2009, pp.~968--977.

\bibitem{boutsidis2014randomized}
{\sc C.~Boutsidis, A.~Zouzias, M.~W. Mahoney, and P.~Drineas}, {\em Randomized
  dimensionality reduction for $ k $-means clustering}, IEEE Transactions on
  Information Theory, 61 (2014), pp.~1045--1062,
  \url{https://doi.org/10.1109/TIT.2014.2375327}.

\bibitem{box2011bayesian}
{\sc G.~E. Box and G.~C. Tiao}, {\em Bayesian inference in statistical
  analysis}, John Wiley \& Sons, 1992,
  \url{https://doi.org/10.1002/9781118033197}.

\bibitem{broadbent2010subset}
{\sc M.~E. Broadbent, M.~Brown, K.~Penner, I.~Ipsen, and R.~Rehman}, {\em
  Subset selection algorithms: Randomized vs. deterministic}, SIAM
  undergraduate research online, 3 (2010).

\bibitem{bui2013computational}
{\sc T.~Bui-Thanh, O.~Ghattas, J.~Martin, and G.~Stadler}, {\em A computational
  framework for infinite-dimensional {B}ayesian inverse problems part {I}: The
  linearized case, with application to global seismic inversion}, SIAM J. Sci.
  Comput, 35 (2013), pp.~A2494--A2523.

\bibitem{businger1965linear}
{\sc P.~Businger and G.~H. Golub}, {\em Linear least squares solutions by
  {H}ouseholder transformations}, Numer. Math., 7 (1965), pp.~269--276.

\bibitem{cai2021robust}
{\sc H.~Cai, K.~Hamm, L.~Huang, and D.~Needell}, {\em Robust cur decomposition:
  Theory and imaging applications}, SIAM Journal on Imaging Sciences, 14
  (2021), pp.~1472--1503, \url{https://doi.org/10.1137/20M1388322}.

\bibitem{chaloner1995bayesian}
{\sc K.~Chaloner and I.~Verdinelli}, {\em Bayesian experimental design: A
  review}, Stat. Sci.,  (1995), pp.~273--304.

\bibitem{chaturantabut2010nonlinear}
{\sc S.~Chaturantabut and D.~C. Sorensen}, {\em Nonlinear model reduction via
  discrete empirical interpolation}, SIAM J. Sci. Comput, 32 (2010),
  pp.~2737--2764.

\bibitem{chowdhary2023pyoed}
{\sc A.~Chowdhary, S.~E. Ahmed, and A.~Attia}, {\em {PyOED}: An extensible
  suite for data assimilation and model-constrained optimal design of
  experiments}, arXiv preprint arXiv:2301.08336,  (2023).

\bibitem{civril2009selecting}
{\sc A.~Civril and M.~Magdon-Ismail}, {\em On selecting a maximum volume
  sub-matrix of a matrix and related problems}, Theor. Comput. Sci., 410
  (2009), pp.~4801--4811.

\bibitem{de2007subset}
{\sc F.~De~Hoog and R.~Mattheij}, {\em Subset selection for matrices}, Linear
  Algebra and its Applications, 422 (2007), pp.~349--359,
  \url{https://doi.org/10.1016/j.laa.2006.08.034}.

\bibitem{de2011note}
{\sc F.~De~Hoog and R.~Mattheij}, {\em A note on subset selection for
  matrices}, Linear algebra and its applications, 434 (2011), pp.~1845--1850,
  \url{https://doi.org/10.1016/j.laa.2010.11.053}.

\bibitem{dong2023simpler}
{\sc Y.~Dong and P.-G. Martinsson}, {\em Simpler is better: a comparative study
  of randomized pivoting algorithms for {CUR} and interpolative
  decompositions}, Adv. Comput. Math., 49 (2023), p.~66.

\bibitem{drineas2006fast}
{\sc P.~Drineas, R.~Kannan, and M.~W. Mahoney}, {\em Fast {M}onte {C}arlo
  algorithms for matrices {I}: Approximating matrix multiplication}, SIAM J.
  Comput., 36 (2006), pp.~132--157.

\bibitem{drmac2018discrete}
{\sc Z.~Drmac and A.~K. Saibaba}, {\em The discrete empirical interpolation
  method: {C}anonical structure and formulation in weighted inner product
  spaces}, SIAM J. Matrix Anal. Appl., 39 (2018), pp.~1152--1180.

\bibitem{duersch2017randomized}
{\sc J.~A. Duersch and M.~Gu}, {\em Randomized {QR} with column pivoting}, SIAM
  J. Sci. Comput, 39 (2017), pp.~C263--C291.

\bibitem{gazzola2019ir}
{\sc S.~Gazzola, P.~C. Hansen, and J.~G. Nagy}, {\em {IR T}ools: a {MATLAB}
  package of iterative regularization methods and large-scale test problems},
  Numer. Algorithms, 81 (2019), pp.~773--811.

\bibitem{golub2012matrix}
{\sc G.~H. Golub and C.~F. Van~Loan}, {\em Matrix computations}, Johns Hopkins
  Studies in the Mathematical Sciences, Johns Hopkins University Press,
  Baltimore, MD, fourth~ed., 2013.

\bibitem{goreinov1997theory}
{\sc S.~A. Goreinov, E.~E. Tyrtyshnikov, and N.~L. Zamarashkin}, {\em A theory
  of pseudoskeleton approximations}, Linear algebra and its applications, 261
  (1997), pp.~1--21, \url{https://doi.org/10.1016/S0024-3795(96)00301-1}.

\bibitem{gu2015subspace}
{\sc M.~Gu}, {\em Subspace iteration randomization and singular value
  problems}, SIAM Journal on Scientific Computing, 37 (2015), pp.~A1139--A1173,
  \url{https://doi.org/10.1137/130938700}.

\bibitem{gu1996efficient}
{\sc M.~Gu and S.~C. Eisenstat}, {\em Efficient algorithms for computing a
  strong rank-revealing {QR} factorization}, SIAM J. Sci. Comput, 17 (1996),
  pp.~848--869.

\bibitem{haber2008numerical}
{\sc E.~Haber, L.~Horesh, and L.~Tenorio}, {\em Numerical methods for
  experimental design of large-scale linear ill-posed inverse problems},
  Inverse Probl., 24 (2008), p.~055012.

\bibitem{halko2011finding}
{\sc N.~Halko, P.-G. Martinsson, and J.~A. Tropp}, {\em Finding structure with
  randomness: Probabilistic algorithms for constructing approximate matrix
  decompositions}, SIAM Rev., 53 (2011), pp.~217--288.

\bibitem{hansen2018air}
{\sc P.~C. Hansen and J.~S. J{\o}rgensen}, {\em {AIR} tools {II}: algebraic
  iterative reconstruction methods, improved implementation}, Numer.
  Algorithms, 79 (2018), pp.~107--137.

\bibitem{herman2020design}
{\sc E.~Herman}, {\em Design of Inverse Problems and Surrogate Modeling in
  Complex Physical Systems}, PhD thesis, North Carolina State University, 2020.

\bibitem{holodnak2015randomized}
{\sc J.~T. Holodnak and I.~C. Ipsen}, {\em Randomized approximation of the gram
  matrix: Exact computation and probabilistic bounds}, SIAM J. Matrix Anal.
  Appl., 36 (2015), pp.~110--137.

\bibitem{horn2013matrix}
{\sc R.~A. Horn and C.~R. Johnson}, {\em Matrix analysis}, Cambridge University
  Press, Cambridge, second~ed., 2013.

\bibitem{huan2024optimal}
{\sc X.~Huan, J.~Jagalur, and Y.~Marzouk}, {\em Optimal experimental design:
  Formulations and computations}, Acta Numerica, 33 (2024), p.~715–840,
  \url{https://doi.org/10.1017/S0962492924000023}.

\bibitem{jagalur2021batch}
{\sc J.~Jagalur-Mohan and Y.~Marzouk}, {\em Batch greedy maximization of
  non-submodular functions: Guarantees and applications to experimental
  design}, J. Mach. Learn. Res., 22 (2021), pp.~11397--11458.

\bibitem{kaipio2005statistical}
{\sc J.~P. Kaipio and E.~Somersalo}, {\em Statistical and Computational Inverse
  Problems}, Springer, 2005, \url{https://doi.org/10.1007/b138659}.

\bibitem{ma2015statistical}
{\sc P.~Ma, M.~W. Mahoney, and B.~Yu}, {\em A statistical perspective on
  algorithmic leveraging}, The Journal of Machine Learning Research, 16 (2015),
  pp.~861--911.

\bibitem{mahoney2009cur}
{\sc M.~W. Mahoney and P.~Drineas}, {\em {CUR} matrix decompositions for
  improved data analysis}, Proc. Natl. Acad. Sci., 106 (2009), pp.~697--702.

\bibitem{martinsson2020randomized}
{\sc P.-G. Martinsson and J.~A. Tropp}, {\em Randomized numerical linear
  algebra: Foundations and algorithms}, Acta Numer., 29 (2020), pp.~403--572.

\bibitem{panteleev2015adjoint}
{\sc G.~Panteleev, M.~Yaremchuk, and W.~E. Rogers}, {\em Adjoint-free
  variational data assimilation into a regional wave model}, J. Atmos. Ocean
  Technol., 32 (2015), pp.~1386--1399.

\bibitem{pukelsheim2006optimal}
{\sc F.~Pukelsheim}, {\em Optimal design of experiments}, vol.~50 of Classics
  in Applied Mathematics, Society for Industrial and Applied Mathematics
  (SIAM), Philadelphia, PA, 2006.
\newblock Reprint of the 1993 original.

\bibitem{quintana1998blas}
{\sc G.~Quintana-Ort{\'\i}, X.~Sun, and C.~H. Bischof}, {\em A {BLAS}-3 version
  of the {QR} factorization with column pivoting}, SIAM J. Sci. Comput, 19
  (1998), pp.~1486--1494.

\bibitem{reich2019bayesian}
{\sc B.~J. Reich and S.~K. Ghosh}, {\em Bayesian statistical methods}, Chapman
  and Hall/CRC, 2019.

\bibitem{saibaba2020randomized}
{\sc A.~K. Saibaba}, {\em Randomized discrete empirical interpolation method
  for nonlinear model reduction}, SIAM J. Sci. Comput, 42 (2020),
  pp.~A1582--A1608.

\bibitem{sorensen2016deim}
{\sc D.~C. Sorensen and M.~Embree}, {\em A {DEIM} induced {CUR} factorization},
  SIAM J. Sci. Comput, 38 (2016), pp.~A1454--A1482.

\bibitem{stuart2010inverse}
{\sc A.~M. Stuart}, {\em Inverse problems: a bayesian perspective}, Acta
  numerica, 19 (2010), pp.~451--559,
  \url{https://doi.org/10.1017/S0962492910000061}.

\bibitem{szyld2006many}
{\sc D.~B. Szyld}, {\em The many proofs of an identity on the norm of oblique
  projections}, Numer. Algorithms, 42 (2006), pp.~309--323.

\bibitem{tarantola2005inverse}
{\sc A.~Tarantola}, {\em Inverse problem theory and methods for model parameter
  estimation}, SIAM, 2005, \url{https://doi.org/10.1137/1.9780898717921}.

\bibitem{tenorio2011data}
{\sc L.~Tenorio, F.~Andersson, M.~De~Hoop, and P.~Ma}, {\em Data analysis tools
  for uncertainty quantification of inverse problems}, Inverse Problems, 27
  (2011), p.~045001.

\bibitem{tsitsvero2016signals}
{\sc M.~Tsitsvero, S.~Barbarossa, and P.~Di~Lorenzo}, {\em Signals on graphs:
  Uncertainty principle and sampling}, IEEE Transactions on Signal Processing,
  64 (2016), pp.~4845--4860, \url{https://doi.org/10.1109/TSP.2016.2573748}.

\bibitem{ucinski2005optimal}
{\sc D.~Uci\'{n}ski}, {\em Optimal measurement methods for distributed
  parameter system identification}, Systems and Control Series, CRC Press, Boca
  Raton, FL, 2005.

\bibitem{voronin2017efficient}
{\sc S.~Voronin and P.-G. Martinsson}, {\em Efficient algorithms for {CUR} and
  interpolative matrix decompositions}, Adv. Comput. Math., 43 (2017),
  pp.~495--516.

\bibitem{wang2019information}
{\sc H.~Wang, M.~Yang, and J.~Stufken}, {\em Information-based optimal subdata
  selection for big data linear regression}, Journal of the American
  Statistical Association, 114 (2019), pp.~393--405.

\bibitem{welch1982algorithmic}
{\sc W.~J. Welch}, {\em Algorithmic complexity: three {NP}-hard problems in
  computational statistics}, J. Stat. Comput. Simul., 15 (1982), pp.~17--25.

\bibitem{wu2023fast}
{\sc K.~Wu, P.~Chen, and O.~Ghattas}, {\em A fast and scalable computational
  framework for large-scale high-dimensional {B}ayesian optimal experimental
  design}, SIAM-ASA J. Uncertain. Quantif., 11 (2023), pp.~235--261.

\bibitem{wu2023offline}
{\sc K.~Wu, P.~Chen, and O.~Ghattas}, {\em An offline-online decomposition
  method for efficient linear {B}ayesian goal-oriented optimal experimental
  design: Application to optimal sensor placement}, SIAM J. Sci. Comput, 45
  (2023), pp.~B57--B77.

\bibitem{zhong2023model}
{\sc W.~Zhong, Y.~Liu, and P.~Zeng}, {\em A model-free variable screening
  method based on leverage score}, Journal of the American Statistical
  Association, 118 (2023), pp.~135--146.

\end{thebibliography}
\end{document}